\documentclass[12pt]{amsart}
\usepackage{latexsym}
\usepackage{amssymb} 
\usepackage{mathrsfs}
\usepackage{amsmath}
\usepackage{latexsym}
\usepackage{delarray}
\usepackage{amssymb,amsmath,amsfonts,amsthm,
mathrsfs}
\usepackage{graphicx}

\usepackage{mathtools}
\usepackage[section]{placeins}

\usepackage{hyperref}  
\renewcommand\eqref[1]{(\ref{#1})} 

\usepackage{xcolor}

\graphicspath{{./Figs/}}

\usepackage{rotating}
\usepackage{wrapfig}

\usepackage{float}
\restylefloat{figure}

\usepackage{subcaption}

\newtheorem{theorem}{Theorem}[section]
\newtheorem{corollary}[theorem]{Corollary}
\newtheorem{lemma}[theorem]{Lemma}
\newtheorem{proposition}[theorem]{Proposition}
\newtheorem{definition}[theorem]{Definition}
\theoremstyle{definition}
\newtheorem{remark}[theorem]{Remark}

\newcommand{\wt}[1]{\widetilde{#1}}

\newcommand{\E}{\ensuremath{{\mathcal E}}}

\newcommand{\mb}[1]{\ensuremath{\mathbb{#1}}}
\newcommand{\N}{\mb{N}}

\newcommand{\R}{\mb{R}}

\renewcommand\N{{\mathbb N}_0}

\newcommand{\beq}{\begin{equation}}
\newcommand{\eeq}{\end{equation}}

\newcommand{\eps}{\varepsilon}

\newcommand{\supp}{\mathrm{supp}}

\title[On the wave equation]
{On the wave equation with space dependent coefficients: singularities and lower order terms}
\author[Marco Discacciati]{Marco Discacciati}
\address{
  Marco Discacciati:
  \endgraf
  Department of Mathematical Sciences
  \endgraf
  Loughborough University
  \endgraf
  Loughborough, Leicestershire, LE11 3TU
  \endgraf
  United Kingdom
  \endgraf
  {\it E-mail address} {\rm m.discacciati@lboro.ac.uk}
  }

\author[Claudia Garetto]{Claudia Garetto}
\address{
  Claudia Garetto:
  \endgraf
School of Mathematical Sciences
  \endgraf
 Queen Mary University of London
  \endgraf
 Mile End Road, London, E1 4NS
  \endgraf
  United Kingdom
  \endgraf
  {\it E-mail address} {\rm c.garetto@qmul.ac.uk}
  }
  
\author[Costas Loizou]{Costas Loizou}
\address{
  Costas Loizou:
  \endgraf
  Department of Mathematical Sciences
  \endgraf
  Loughborough University
  \endgraf
  Loughborough, Leicestershire, LE11 3TU
  \endgraf
  United Kingdom
  \endgraf  {\it E-mail address} {\rm c.loizou@lboro.ac.uk}
  }

\thanks{The second author was supported by the
EPSRC grant  EP/V005529/2}

\date{}

\subjclass[2010]{Primary 35L05: 35L10; Secondary 35D99;}
\keywords{Hyperbolic equations, very weak solutions, regularisation}

\begin{document}

\maketitle

\begin{abstract}
This paper complements the study of the wave equation with discontinuous coefficients initiated in \cite{DGL:22} in the case of time-dependent coefficients. Here we assume that the equation coefficients are depending on space only and we formulate Levi conditions on the lower order terms to guarantee the existence of a very weak solution as defined in \cite{GR:14}. As a toy model we study the wave equation in conservative form with discontinuous velocity and we provide a qualitative analysis of the corresponding very weak solution via numerical methods.
\end{abstract}

\section{Introduction}
In this paper we want to study the well-posedness of the Cauchy problem for the inhomogeneous wave equation with space-dependent coefficients. In detail, we are concerned with 
\beq
\label{CP_1_intro}
\begin{split}
\partial_t^2u-a(x)\partial^2_{x} u+ b_1(x)\partial_x u+ b_2(x)\partial_t u + b_3(x) u&=f(t,x),\quad t\in[0,T],\, x\in\R,\\
u(0,x)&=g_0,\\
\partial_tu(0,x)&=g_1,
\end{split}
\eeq
where $a(x)\ge 0$ and for the sake of simplicity we work in space dimension 1. The well-posedness of \eqref{CP_1_intro} is well-understood when the coefficients are regular, namely smooth. Indeed,  the equation above can be re-written in the variational form
\[
\begin{split}
u_{tt}-(a(x)u_{x})_x+ (a'(x)+b_1(x))u_x+ b_2(x)u_t + b_3(x) u&=f(t,x),\quad t\in[0,T],\, x\in\R,\\
u(0,x)&=g_0,\\
\partial_tu(0,x)&=g_1.
\end{split}
\]
This kind of Cauchy problem has been studied by Oleinik in \cite{O70}. Assuming that the coefficients are smooth and bounded, with bounded derivatives of any order, she proved that the Cauchy problem is $C^\infty$ well-posed provided that  the following Oleinik's condition is satisfied\footnote{There exists a constant $D>0$ such that $(a'(x)+b_1(x))^2 \le   D a(x)$ for all $x\in\R$.} :
\[
(a'+b_1)^2 \prec a.
\]
Note that $a'$ is bounded by $\sqrt{a}$ as a direct consequence of  Glaeser's inequality: If $a\in C^2(\R)$, $a(x)\ge 0$ for all $x\in\R$ and $\Vert a'' \Vert_{L^\infty}\le M_1$, then 
\[
|a'(x)|^2\le 2M_1 a(x),
\]
for all $x\in\R$.

Therefore, $C^\infty$ well-posedness is obtained by simply imposing on the lower order term $b_1$ a Levi condition of the type
\beq	 \label{Levi}
b_1^2(x) \le M_2 a(x),
\eeq
for some $M_2>0$ independent of $x$.
It is of physical interest to understand the well-posedness of this Cauchy problem when the coefficients are less than continuous. This kind on investigation has been initiated in \cite{GR:14} for second order hyperbolic equations with $t$-dependent coefficients and recently extended to inhomogeneous equations in \cite{DGL:22}. Here we want to work with space-dependent coefficients with minimal assumptions of regularity, namely distributions with compact support. We are motivated by the toy model  
\beq
\label{CP_heaviside}
\begin{split}
\partial_t^2u(t,x)-\partial_x (H(x) \partial_x u(t,x))&=0,\qquad t\in[0,T],\, x\in\R,\\
u(0,x)&=g_0(x),\qquad x\in\R,\\
\partial_t u(0,x)&=g_1(x), \qquad x\in\R,
\end{split}
\eeq
where $g_0,\, g_1 \in C_c^{\infty}(\R)$ and $H$ is the Heaviside function ($H(x)=0$ if $x<0$, $H(x)=1$ if $x\ge 0$) or more in general $H$ is replaced by a positive distribution with compact support. Note that the well-posedness of the Cauchy problem for hyperbolic equations has been widely investigated when the equation coefficients are at least continuous, see \cite{ColKi:02, ColKi:02-2, CS, CDS, GR:11, GR:12} and references therein. However, in presence of discontinuities distributional solutions might fail to exists due to the well-known Schwartz impossibility result \cite{Schwartz:impossibility-1954}.

For this reason, as in \cite{G:21, GR:14} we look for solutions of the Cauchy problem \eqref{CP_1_intro} in the \emph{very weak} sense. In other words we replace the equation under consideration with a family of regularised equations obtained via convolution with a net of mollifiers. We will then obtain a net $(u_\eps)_\eps$ that we will analyse in terms of qualitative and limiting behaviour as $\eps\to 0$. 
The paper is organised as follows.

In Section 2 we revisit Oleinik's result in the case of smooth coefficients. We show how her condition on the lower order term $b_1$ can be obtained via transformation into a first order system and energy estimates. Note that the energy is provided by the hyperbolic symmetriser associated to the wave operator. This system approach turns out to be easily adaptable to the case of singular coefficients and in general to the framework of very weak solutions. In Section 3 we pass to consider discontinuous coefficients. After a short introduction to the notion of very weak solution, we formulate Levi conditions on the lower order terms which allow to prove that our Cauchy problem admits a very weak solution of Sobolev type. Some toy models are analysed in Section 4 where we prove that every weak solution to \eqref{CP_heaviside} recovers, in the limit as $\eps\to 0$, the piecewise distributional solution defined in \cite{DO:16}. More singular toy models, defined via delta of Dirac and homogenous distributions,  are also presented in Section 4 and the corresponding very weak solutions investigated numerically.

\section{A revisited approach to Oleinik's result}
\label{sec_Oleinik}
This section is devoted to the Cauchy problem
\beq
\label{CP_1_ex}
\begin{split}
\partial_t^2u-a(x)\partial^2_{x} u+ b_1(x)\partial_x u+ b_2(x)\partial_t u + b_3(x) u&=f(t,x),\quad t\in[0,T],\, x\in\R,\\
u(0,x)&=g_0,\\
\partial_tu(0,x)&=g_1,
\end{split}
\eeq
where $a, b_1, b_2, b_3\in B^{\infty}(\R)$, are smooth and bounded with bounded derivatives of any order, and $a \geq 0$. We also assume that all the functions involved in the system are real-valued. The $C^\infty$ well-posedness of \eqref{CP_1_ex} is known thanks to \cite{O70}. Here, we give an alternative proof of this result based on the reduction to a first order system.

\subsection{System in U} \label{section_classical}
In detail, by using the transformation,
\[
U=(U^0,U_1,U_2)=(u,\partial_{x}u,\partial_tu)^{T},
\]
our Cauchy problem can be rewritten as
\[
\begin{split}
\partial_t U&=A \partial_x U+ B U+F,\\
U(0,x)&=(g_0,g_0',g_1)^{T},
\end{split}
\]
where
\[
A=\left(
	\begin{array}{ccc}
	0& 0 & 0\\
	0& 0 & 1\\
    0& a & 0 
	\end{array}
	\right),
	\quad
B=\left(
	\begin{array}{ccc}
	0& 0 & 1\\
	0& 0 & 0\\
    -b_3 &-b_1 & -b_2 
	\end{array}
	\right) 
	\text{ and }	\quad	
F=\left(
	\begin{array}{c}
	0\\
	0 \\
	f
        \end{array}
	\right).
	\]

The matrix $A$  has a block diagonal shape with a $1\times1$ block equal to $0$ and a $2\times2$ block  in Sylvester form and has the symmetriser 
\[
Q=\left(
	\begin{array}{ccc}
	1&0&0 \\
	0&a& 0\\
    0&0 & 1 
	\end{array}
	\right),
\]
i.e., $QA=A^\ast Q=A^tQ$. The symmetriser $Q$ defines the energy  
\[
E(t)=(QU,U)_{L^2}.
\]
Since $a\ge 0$, we have that the bound from below
\[
E(t)=\Vert U^0\Vert^2_{L^2}+(aU_1,U_1)_{L^2}+\Vert U_2\Vert^2_{L^2}\ge \Vert U^0\Vert^2_{L^2}+\Vert U_2\Vert^2_{L^2}
\]
holds, for all $t\in[0,T]$. Assume that the initial data $g_0, g_1$ are compactly supported and that $f$ is compactly supported with respect to $x$. By the finite speed of propagation, it follows that that the solution $U$ is compactly supported with respect to $x$ as well. Hence, by integration by parts we obtain the following energy estimate:
 \[
 \begin{split}
 &\frac{dE(t)}{dt}=(\partial_t(QU),U)_{L^2}+(QU,\partial_tU)_{L^2}\\
 &=(Q\partial_tU,U)_{L^2}+(QU, A\partial_{x}U)_{L^2}+(QU, BU)_{L^2}+(QU,F)_{L^2}\\
 &=(QA\partial_{x}U, U)_{L^2}+(QBU, U)_{L^2}+(QU, A\partial_{x}U)_{L^2}+(QU, BU)_{L^2}+2(QU,F)_{L^2}\\
 &= (QA\partial_{x}U, U)_{L^2}+(QBU, U)_{L^2}+(A^*QU, \partial_{x}U)_{L^2}+(B^*QU,U)_{L^2}+2(QU,F)_{L^2}\\
 &= (QA\partial_{x}U, U)_{L^2}+(QAU, \partial_{x}U)_{L^2}+(QBU, U)_{L^2}+(B^*QU,U)_{L^2}+2(QU,F)_{L^2}\\
 &=(QA\partial_{x}U,U)_{L^2}-(\partial_{x}(QAU),U)_{L^2}+((QB+B^*Q)U, U)_{L^2}+2(QU,F)_{L^2}\\
 &=-((QA)'U,U)_{L^2}+((QB+B^*Q)U, U)_{L^2}+2(QU,F)_{L^2}.
\end{split}
 \]
Since 
\[
(QA)'=\left(
	\begin{array}{ccc}
	0&0&0\\
	0&0& a'\\
    0&a' & 0
	\end{array}
	\right),
\]
from Glaeser's inequality ($|a'(x)|^2\le 2M_1 a(x)$) it immediately follows that 
\begin{align}\label{estimate1_3x3}
((QA)'U,U)_{L^2}&=2(a'U_1,U_2)_{L^2}\le 2\Vert a'U_1\Vert_{L^2}\Vert U_2\Vert_{L^2}\le \Vert a'U_1\Vert_{L^2}^2+\Vert U_2\Vert^2_{L^2}\\
&\le 2M_1(aU_1,U_1)_{L^2}+\Vert U_2\Vert^2_{L^2}
\le \max(2M_1,1)E(t). \nonumber
\end{align}
Furthermore we have,
\[
QB+B^*Q=\left(
	\begin{array}{ccc}
	0     & 0   & 1-b_3\\
	0     & 0   & -b_1\\
    1-b_3 &-b_1 & -2b_2
	\end{array}
	\right).
\]
Hence, using the Levi condition \eqref{Levi}, $b_1^2(x)\le M_2 a(x)$, we have
\begin{align} \label{estimate2_3x3}
((QB+B^*Q)U, U)_{L^2}&=2((1-b_3)U_2,U^0)- 2(b_1 U_1,U_2)_{L^2}- 2(b_2 U_2,U_2)_{L^2} \nonumber\\
&\le 2\Vert (1-b_3)U_2\Vert_{L^2}\Vert U^0\Vert_{L^2}+ 2\Vert b_1 U_1\Vert_{L^2}\Vert U_2\Vert_{L^2}+ 2\Vert b_2\Vert_\infty \Vert U_2\Vert_{L^2}^2  \nonumber\\
&\le \Vert (1-b_3)\Vert_\infty(\Vert U_2\Vert_{L^2}^2+\Vert U^0\Vert_{L^2}^2)+\Vert b_1 U_1\Vert_{L^2}^2+(1+2\Vert b_2\Vert_\infty)\Vert U_2\Vert_{L^2}^2 \nonumber\\
&= \Vert (1-b_3)\Vert_\infty(\Vert U_2\Vert_{L^2}^2+\Vert U^0\Vert_{L^2}^2)+(|b_1|^2  U_1,U_1)+(1+2\Vert b_2\Vert_\infty)\Vert U_2\Vert_{L^2}^2  \nonumber \\
& \le (1+\Vert b_3\Vert_\infty)\Vert U^0\Vert_{L^2}^2+ M_2(aU_1,U_1)+(2+2\Vert b_2\Vert_\infty+\Vert b_3\Vert_\infty)\Vert U_2\Vert_{L^2}^2 \nonumber\\
& \le \max(M_2,1+\Vert b_3\Vert_\infty, 2+2\Vert b_2\Vert_\infty+\Vert b_3\Vert_\infty)E(t). 
\end{align}
Finally,
\beq \label{estimate3_3x3}
2(QU,F)_{L^2}=2(U_2,f)_{L^2}\le 2\Vert U_2\Vert_{L^2}\Vert f\Vert_{L^2} \le\Vert U_2\Vert_{L^2}^2+\Vert f\Vert_{L^2}^2 \le E(t)+\Vert f\Vert_{L^2}.
\eeq
Combining \eqref{estimate1_3x3}, \eqref{estimate2_3x3} and \eqref{estimate3_3x3}, we obtain the estimate 
\[
\frac{dE}{dt}\le cE(t)+\Vert f\Vert_{L^2},
\]
where

\begin{align*}
c&=\max(2M_1+1,2)+\max(M_2,1+\Vert b_3\Vert_\infty, 2+2\Vert b_2\Vert_\infty+\Vert b_3\Vert_\infty)\\
&=\max(\max(2M_1+1+M_2, 2M_1+2+\Vert b_3\Vert_\infty, 2M_1+3+2\Vert b_2\Vert_\infty+\Vert b_3\Vert_\infty),\\
 &\quad \quad\max( 2+M_2,3+\Vert b_3\Vert_\infty, 4+2\Vert b_2\Vert_\infty+\Vert b_3\Vert_\infty))\\
&=\max(2M_1+1+M_2, 2M_1+2+\Vert b_3\Vert_\infty, 2M_1+3+2\Vert b_2\Vert_\infty+\Vert b_3\Vert_\infty,\\
 &\quad \quad \quad \quad 2+M_2,3+\Vert b_3\Vert_\infty, 4+2\Vert b_2\Vert_\infty+\Vert b_3\Vert_\infty)\\
&=\max(2M_1+1+M_2,  2M_1+3+2\Vert b_2\Vert_\infty+\Vert b_3\Vert_\infty,2+M_2,4+2\Vert b_2\Vert_\infty+\Vert b_3\Vert_\infty) .
\end{align*}
Using the bound from below for the energy and Gr\"onwall's lemma we obtain the following estimate for $U^0$ and $U_2$:  
\[
\Vert U^0(t)\Vert_{L^2}^2+\Vert U_{2}(t)\Vert_{L^2}^2\le E(t)\le \biggl(E(0)+\int_{0}^t\Vert f(s)\Vert_{L^2}^2\, ds\biggr){\rm e}^{ct}.
\]
In addition,
\[
\begin{split}
\biggl(E(0)+\int_{0}^t\Vert f(s)\Vert_{L^2}^2\, ds\biggr){\rm e}^{ct}
&\le {\rm e}^{cT}\Vert U^0(0)\Vert_{L^2}^2+ {\rm e}^{cT}\Vert a\Vert_{\infty}\Vert U_1(0)\Vert_{L^2}^2+ {\rm e}^{cT}\Vert U_2(0)\Vert_{L^2}^2\\
& +{\rm e}^{cT}\int_{0}^t\Vert f(s)\Vert_{L^2}^2\, ds\\
&\le C_{2}\biggl(\Vert g_0\Vert_{H^1}^2+\Vert g_1\Vert_{L^2}^2+\int_{0}^t\Vert f(s)\Vert_{L^2}^2\, ds\biggr),
\end{split}
\]
for all $t\in[0,T]$. Note that the constant $C_2$ depends linearly on $\Vert a\Vert_\infty$ and exponentially on $T$, $M_2$,  $\Vert a''\Vert_\infty$ $\Vert b_2\Vert_\infty$ and $\Vert b_3\Vert_\infty$. Indeed, setting $M_1=\Vert a''\Vert_\infty$,
\beq \label{C_2_3x3}
C_2= {\rm e}^{cT}\max(\Vert a\Vert_{\infty},1)= {\rm e}^{\max(2M_1+1+M_2,  2M_1+3+2\Vert b_2\Vert_\infty+\Vert b_3\Vert_\infty,2+M_2,4+2\Vert b_2\Vert_\infty+\Vert b_3\Vert_\infty) T}\max(\Vert a\Vert_{\infty},1).
\eeq
Concluding,
\beq \label{est_U_2_3x3}
\Vert U^0(t)\Vert_{L^2}^2+\Vert U_{2}(t)\Vert_{L^2}^2\le C_{2}\biggl(\Vert g_0\Vert_{H^1}^2+\Vert g_1\Vert_{L^2}^2+\int_{0}^t\Vert f(s)\Vert_{L^2}^2\, ds\biggr).
\eeq

\subsection{$L^2$-estimates for $U_1$}
We now want to obtain a similar estimate for $U_1$. To  attain this, we transform once more the system by taking a derivative with respect to $x$. Let $V=(V^0,V_1,V_2)^T=(\partial_{x}U^0,\partial_{x}U_1, \partial_{x}U_2)^T$. By getting an estimate for $V_2$ we also automatically  get an estimate for $U_1$ since $V_2=\partial_t U_1$. Indeed, we can do so by applying the fundamental theorem of calculus and making use of the initial conditions. Hence, if $U$ solves
\[
\begin{split}
\partial_t U=&A \partial_x U+BU+F,\\
U(0,x)&=(g_0,g'_0,g_1)^{T},
\end{split}
\]
then $V$ solves
\[
\begin{split}
\partial_t V=& A\partial_x V+(A'+ B)V +\tilde{F},\\
V(0,x)&=(g'_0, g_0'',g_1')^{T},\quad \text{ where } \tilde{F} =B'U+F_x. 
\end{split}
\]
The system in $V$  still has $A$ as a principal part matrix and additional lower order terms. It follows that we can still use the symmetriser $Q$ to define the energy
\[
E(t)=(QV,V)_{L^2}=\Vert V^0\Vert^2_{L^2}+(aV_1,V_1)_{L^2}+\Vert V_2\Vert^2_{L^2},
\]
for which we obtain the bound from below $ \Vert V^0\Vert^2_{L^2}+\Vert V_2\Vert^2_{L^2} \le E(t)$. Therefore,
 \[
 \begin{split}
 &\frac{dE(t)}{dt}=(\partial_t(QV),V)_{L^2}+(QV,\partial_tV)_{L^2}\\
 &=(Q\partial_tV,V)_{L^2}+(QV, A\partial_x V+(A'+ B)V )_{L^2}+(QV,\tilde{F})_{L^2}\\
 &=(QA\partial_x V+Q(A'+ B)V +Q\tilde{F}, V)_{L^2}+(QV, A\partial_{x}V+(A'+B)V)_{L^2}+(QV,\tilde{F})_{L^2}\\
 &= (QA\partial_{x}V,V)_{L^2}+(A^\ast QV,\partial_{x}V)_{L^2}+(Q(A'+B)V, V)_{L^2}+(QV, (A'+B)V)_{L^2}+2(QV,\tilde{F})_{L^2}\\
 &=(QA\partial_{x}V,V)_{L^2}+(QAV,\partial_{x}V)_{L^2}+2(Q(A'+B)V, V)_{L^2}+2(QV,\tilde{F})_{L^2}\\
 &=(QA\partial_{x}V,V)_{L^2}-(\partial_{x}(QAV),V)_{L^2}+2(Q(A'+B)V, V)_{L^2}+2(QV,\tilde{F})_{L^2}\\
 &=-((QA)'V,V)_{L^2}+ 2(Q(A'+B)V, V)_{L^2}+2(QV,\tilde{F})_{L^2}.
\end{split}
 \]
By direct computations
\begin{align*}
 ((QA)'V,V)_{L^2}&=2(a'V_1,V_2)_{L^2}\\
 2(Q(A'+B)V, V)_{L^2}&=2((1-b_3)V_2,V^0)_{L^2}+2(a'V_1,V_2)_{L^2}-2(b_1V_1,V_2)_{L^2}-2(b_2V_2,V_2)_{L^2}\\
 2(QV,\tilde{F})_{L^2}&=-2(b'_3U^0,V_2)_{L^2}-2(b'_1U_1,V_2)_{L^2}-2(b'_2U_2,V_2)_{L^2}+2(V_2,f_x)_{L^2}.
\end{align*}
 Hence,
 \begin{align*}
 \frac{dE(t)}{dt}=&2((1-b_3)V_2,V^0)_{L^2}+2(V_2,f_x)_{L^2}-2(b_1V_1,V_2)_{L^2}-2(b_2V_2,V_2)_{L^2}\\
 &-2(b'_3U^0,V_2)_{L^2}-2(b'_1U_1,V_2)_{L^2}-2(b'_2U_2,V_2)_{L^2} .
 \end{align*}
 Now,
\begin{align*}
&|2((1-b_3)V_2,V^0)_{L^2} | \le 2(1+\Vert b_3 \Vert_\infty )( \Vert V_2\Vert_{L^2}^2+ \Vert V^0\Vert_{L^2}^2)\le  2(1+\Vert b_3 \Vert_\infty ) E(t)\\
&|2(V_2,f_x)_{L^2}|\le \Vert V_2\Vert_{L^2}^2+ \Vert f_x\Vert_{L^2}^2\le E(t)+\Vert f_x\Vert^2_{L^2}\\
&|2(b_1V_1,V_2)_{L^2}|\le \Vert b_1 V_1\Vert_{L^2}^2+ \Vert V_2\Vert_{L^2}^2 \le M_2(aV_1,V_1)_{L^2}+\Vert V_2\Vert^2_{L^2}\le \max(M_2,1) E(t),\\
&|2(b_2V_2,V_2)_{L^2}| \le 2\Vert b_2 \Vert_\infty \Vert V_2\Vert_{L^2}^2 \le  2\Vert b_2 \Vert_\infty E(t)\\
&|2(b'_1U_1,V_2)_{L^2}| \le \Vert b'_1 \Vert^2_\infty \Vert U_1\Vert_{L^2}^2+ \Vert V_2\Vert_{L^2}^2 \le  \Vert b'_1 \Vert^2_\infty \Vert U_1\Vert_{L^2}^2+ E(t)
\end{align*}
and using \eqref{est_U_2_3x3},
\begin{align*}
&|2(b'_3U^0,V_2)_{L^2}| \le \Vert b'_3 \Vert^2_\infty \Vert U^0\Vert_{L^2}^2+ \Vert V_2\Vert_{L^2}^2\le C_{2} \Vert b'_3 \Vert^2_\infty \biggl(\Vert g_0\Vert_{H^1}^2+\Vert g_1\Vert_{L^2}^2+\int_{0}^t\Vert f(s)\Vert_{L^2}^2\, ds\biggr) +E(t)\\
&|2(b'_2U_2,V_2)_{L^2}| \le \Vert b'_2 \Vert^2_\infty \Vert U_2\Vert_{L^2}^2+ \Vert V_2\Vert_{L^2}^2\le C_{2} \Vert b'_2 \Vert^2_\infty \biggl(\Vert g_0\Vert_{H^1}^2+\Vert g_1\Vert_{L^2}^2+\int_{0}^t\Vert f(s)\Vert_{L^2}^2\, ds\biggr) +E(t).
\end{align*}
Therefore, 
\begin{align}\label{Energy_derivative1_3x3}
 \frac{dE(t)}{dt}\le&(6+2\Vert b_3 \Vert_\infty+2\Vert b_2 \Vert_\infty+\max(M_2,1))E(t) +\Vert f_x(t)\Vert_{L^2}^2+\Vert b'_1 \Vert^2_\infty \Vert U_1\Vert_{L^2}^2\\
 				&+C_{2} (\Vert b'_2 \Vert^2_\infty+\Vert b'_3 \Vert^2_\infty) \biggl(\Vert g_0\Vert_{H^1}^2+\Vert g_1\Vert_{L^2}^2+\int_{0}^t\Vert f(s)\Vert_{L^2}^2\, ds\biggr). \nonumber
\end{align}

Now we note that $V_2=\partial_x U_2=\partial_x \partial_t u=\partial_t \partial_x u=\partial_t U_1$. By the fundamental theorem of calculus we have
\[
  \Vert U_1(t)\Vert^2_{L^2}\le 2\Vert U_1(t)-U_1(0)\Vert_{L^2}^2+2\Vert U_1(0)\Vert_{L^2}^2
 =2\Big\Vert \int_{0}^tV_{2}\, ds\Big\Vert_{L^2}^2+2\Vert U_1(0)\Vert_{L^2}^2.\\
 \]
  By Minkowski's integral inequality
  \[
  \Big\Vert \int_{0}^tV_{2}\, ds\Big\Vert_{L^2}\le \int_{0}^t \Vert V_2(s)\Vert_{L^2}ds
  \]
  and therefore by applying Hölder's inequality on the integral in $ds$ we get
  \begin{align}\label{U_1_3x3}
    \Vert U_1(t)\Vert^2_{L^2}&\le 2\biggl(\int_{0}^t \Vert V_{2}(s)\Vert_{L^2}ds\biggr)^2+2\Vert U_1(0)\Vert_{L^2}^2\le 2t \int_{0}^t \Vert V_{2}(s)\Vert^2_{L^2}ds+2\Vert U_1(0)\Vert_{L^2}^2\\
    &\le 2T \int_{0}^t E(s)ds+2\Vert U_1(0)\Vert_{L^2}^2. \nonumber
  \end{align}
 Hence, estimate \eqref{Energy_derivative1_3x3} becomes
\beq\label{Energy_derivative2_3x3}
\begin{split}
&\frac{dE(t)}{dt}\le(6+2\Vert b_3 \Vert_\infty+2\Vert b_2 \Vert_\infty+\max(M_2,1))E(t) +2T \Vert b'_1 \Vert^2_\infty  \int_{0}^t E(s)ds+\Vert f_x(t)\Vert_{L^2}^2\\
&+2\Vert b'_1 \Vert^2_\infty\Vert U_1(0)\Vert_{L^2}^2 +C_{2} (\Vert b'_2 \Vert^2_\infty+\Vert b'_3 \Vert^2_\infty) \biggl(\Vert g_0\Vert_{H^1}^2+\Vert g_1\Vert_{L^2}^2+\int_{0}^t\Vert f(s)\Vert_{L^2}^2\, ds\biggr).\nonumber 	
 \end{split}			
\eeq
For the sake of the reader we now recall a Gr\"onwall's type lemma (Lemma 6.2 in \cite{ST}) that will be applied to the inequality \eqref{Energy_derivative2_3x3} in order to estimate the energy.
\begin{lemma}
\label{lem_ST_G_3x3}
Let $\varphi\in C^1([0,T])$ and $\psi\in C([0,T])$ two positive functions such that 
\[
\varphi'(t)\le B_1\varphi(t)+B_2\int_0^t\varphi(s)\, ds+\psi(t),\qquad t\in[0,T],
\]
for some constants $B_1,B_2>0$. Then, there exists a constant $B>0$ depending exponentially on $B_1, B_2$ and $T$ such that 
\[
\varphi(t)\le B\biggl(\varphi(0)+\int_0^t\psi(s)\, ds\biggl),
\]
for all $t\in[0,T]$.
\end{lemma}
Hence, combining the bound from below for the energy $E(t)$ with Lemma \eqref{lem_ST_G_3x3}, we obtain that there exists a constant $C_3>0$ depending exponentially on $\Vert b_2 \Vert_\infty$, $\Vert b_3 \Vert_\infty$, $\Vert b'_1 \Vert^2_\infty$, $M_2$ and $T$ such that
\begin{align}
\label{est_V_2_3x3}
&\Vert V^0(t)\Vert_{L^2}^2+\Vert V_{2}(t)\Vert_{L^2}^2 \le E(t)\le C_3\biggl(E(0)+\int_0^t \Vert f_x(s)\Vert_{L^2}^2+ C_{2} (\Vert b'_2 \Vert^2_\infty+\Vert b'_3 \Vert^2_\infty) \int_{0}^s\Vert f(r)\Vert_{L^2}^2\, dr\, ds \\
&+2T\Vert b'_1 \Vert^2_\infty \Vert U_1(0)\Vert_{L^2}^2+C_{2}T(\Vert b'_2 \Vert^2_\infty+\Vert b'_3 \Vert^2_\infty)\left(\Vert g_0\Vert_{H^1}^2+\Vert g_1\Vert_{L^2}^2\right)\biggl) \nonumber \\
\le& C_3\biggl(\Vert V^0(0)\Vert_{L^2}^2+ \Vert a\Vert_\infty \Vert V_1(0)\Vert_{L^2}^2+\Vert V_2(0)\Vert_{L^2}^2
+ C_{2}T (\Vert b'_2 \Vert^2_\infty+\Vert b'_3 \Vert^2_\infty) \int_{0}^t\Vert f(s)\Vert_{L^2}^2\, ds\nonumber\\
&+\int_0^t \Vert f_x(s)\Vert_{L^2}^2\, ds+2T\Vert b'_1 \Vert^2_\infty \Vert g_1\Vert_{L^2}^2+C_{2}T (\Vert b'_2 \Vert^2_\infty+\Vert b'_3 \Vert^2_\infty) \left(\Vert g_0\Vert_{H^1}^2+\Vert g_1\Vert_{L^2}^2\right)\biggl)  \nonumber\\
\le& C_3\biggl( \Vert g_0\Vert_{H^1}^2+\Vert a\Vert_\infty \Vert g_0\Vert_{H^2}^2+\Vert g_1\Vert_{H^1}^2+ \max( C_{2}T (\Vert b'_2 \Vert^2_\infty+\Vert b'_3 \Vert^2_\infty) ,1) \int_{0}^t\Vert f(s)\Vert_{H^1}^2\, ds\nonumber\\
&+2T\Vert b'_1 \Vert^2_\infty\Vert g_1\Vert_{L^2}^2+C_{2}T(\Vert b'_2 \Vert^2_\infty+\Vert b'_3 \Vert^2_\infty) \left(\Vert g_0\Vert_{H^1}^2+\Vert g_1\Vert_{L^2}^2\right)\biggl) \nonumber\\
\le& C_3 \max(\Vert a\Vert_\infty,  C_{2} T\Vert b'_2 \Vert^2_\infty, C_{2} T\Vert b'_3 \Vert^2_\infty,1,2T\Vert b'_1 \Vert^2_\infty)\biggl(\Vert g_0\Vert_{H^2}^2+\Vert g_1\Vert_{H^1}^2+\int_{0}^t\Vert f(s)\Vert_{H^1}^2 \, ds\biggl),\nonumber
\end{align}
for all $t\in[0,T]$, with $C_2$  as in \eqref{C_2_3x3}. Noting that $V^0=U_1$, we have that
\begin{align} \label{est_U_1_3x3}
\Vert U_1(t)\Vert^2_{L^2}\le&   C_3 \max(\Vert a\Vert_\infty,  C_{2} T\Vert b'_2 \Vert^2_\infty, C_{2} T\Vert b'_3 \Vert^2_\infty,1,2T\Vert b'_1 \Vert^2_\infty) \biggl(\Vert g_0\Vert_{H^2}^2+\Vert g_1\Vert_{H^1}^2+\int_{0}^t\Vert f(s)\Vert_{H^1}^2 \, ds\biggl).
\end{align}

\subsubsection{System in W}
Analogously, if we want to estimate the $L^2$-norm of $V_1$ we need to repeat the same procedure, i.e.,  to derive the system in $V$ with respect to $x$ and introduce $W=(W^0,W_1,W_2)^T=(\partial_{x}V^0,\partial_{x}V_1, \partial_{x}V_2)^T$. We have that if $V$ solves 
\[
\begin{split}
\partial_t V=& A\partial_x V+(A'+ B)V +\tilde{F},\\
V(0,x)&=(g'_0,g_0'',g_1')^{T},\quad \text{ where } \tilde{F} =B'U+F_x, 
\end{split}
\]
then $W$ solves
\[
\begin{split}
\partial_t W=& A \partial_x W+(2A'+B)W+\tilde{\tilde{F}},\\
W(0,x)&=(g_0',g_0''',g_1'')^{T},\quad \text{ where } \tilde{\tilde{F}} =(A''+2B')V+B''U+F_{xx}.
\end{split}
\]
Again, by using the energy $E(t)=(QW, W)_{L^2}=\Vert W^0\Vert^2_{L^2}+(aW_1,W_1)_{L^2}+\Vert W_2\Vert^2_{L^2}$ we have
 \[
 \begin{split}
 &\frac{dE(t)}{dt}=(\partial_t(QW),W)_{L^2}+(QW\partial_tW)_{L^2}\\
 &=(Q\partial_tW,W)_{L^2}+(QW,A \partial_x W+(2A'+B)W)_{L^2}+(QW,\tilde{\tilde{F}})_{L^2}\\
 &=(QA \partial_x W+Q(2A'+B)W+Q\tilde{\tilde{F}}, W)_{L^2}+(QW,A \partial_x W+(2A'+B)W)_{L^2}+(QW,\tilde{\tilde{F}})_{L^2}\\
 &=(QA \partial_x W, W)_{L^2}+(Q(2A'+B)W, W)_{L^2}+(A^*QW, \partial_x W)_{L^2}+(QW,(2A'+B)W)_{L^2}+2(QW,\tilde{\tilde{F}})_{L^2}\\
 &=(QA \partial_x W, W)_{L^2}+(QAW, \partial_x W)_{L^2}+2(Q(2A'+B)W, W)_{L^2}+2(QW,\tilde{\tilde{F}})_{L^2}\\
 &=(QA \partial_x W, W)_{L^2}-(\partial_x(QAW), W)_{L^2}+2(Q(2A'+B)W, W)_{L^2}+2(QW,\tilde{\tilde{F}})_{L^2}\\
 &=-((QA)'W,W)_{L^2}+2(Q(2A'+B)W, W)_{L^2}+2(QW,\tilde{\tilde{F}})_{L^2}.
\end{split}
 \]
By direct computations we get
\begin{align*}
 ((QA)'W,W)_{L^2}=&2(a'W_1,W_2)_{L^2}\\
 2(Q(2A'+B)W, W)_{L^2}=&2((1-b_3)W^0,W_2)_{L^2}+4(a'W_1,W_2)_{L^2}-2(b_1W_1,W_2)_{L^2}-2(b_2W_2,W_2)_{L^2}\\
 2(QW,\tilde{F})_{L^2}=&2(a''W_2,V_1)_{L^2}-4(b'_3W_2,V^0)_{L^2}-4(b'_1W_2,V_1)_{L^2}-4(b'_2W_2,V_2)_{L^2}\\
 &-2(b''_3W_2,U^0)_{L^2}-2(b''_1W_2,U_1)_{L^2}-2(b''_2W_2,U_2)_{L^2}+2(W_2,f_{xx})_{L^2}.
\end{align*}
 Hence,
 \begin{align*}
 \frac{dE(t)}{dt}=&2(a'W_1,W_2)_{L^2}+2((1-b_3)W_2,W^0)_{L^2}-2(b_1W_1,W_2)_{L^2}-2(b_2W_2,W_2)_{L^2}\\
 &+2(a''W_2,V_1)_{L^2}-4(b'_1W_2,V_1)_{L^2}-4(b'_2W_2,V_2)_{L^2}-2(b''_3W_2,U^0)_{L^2}\\
 &-2(b''_1W_2,U_1)_{L^2}-2(b''_2W_2,U_2)_{L^2}+2(W_2,f_{xx})_{L^2}.
 \end{align*}
 Now,
\begin{align*}
&|2(a'W_1,W_2)_{L^2}|\le 2\Vert a'W_1\Vert_{L^2}\Vert W_2\Vert_{L^2}\le 2M_1(aW_1,W_1)_{L^2}+\Vert W_2\Vert_{L^2}^2\le \max(2M_1,1)E(t)\\
&|2((1-b_3)W^0,W_2)_{L^2}| \le (1+\Vert b_3\Vert_\infty) (\Vert W^0\Vert^2_{L^2}+\Vert W_2\Vert_{L^2}^2)\le  (1+\Vert b_3\Vert_\infty)E(t)\\
&|2(b_1W_1,W_2)_{L^2}|\le 2\Vert b_1W_1\Vert_{L^2}\Vert W_2\Vert_{L^2}\le M_2(aW_1,W_1)_{L^2}+\Vert W_2\Vert_{L^2}^2\le \max(M_2,1)E(t) \\
&|2(b_2W_2,W_2)_{L^2}| \le 2\Vert b_2 \Vert_\infty \Vert W_2\Vert_{L^2}^2 \le  2\Vert b_2 \Vert_\infty E(t)\\
&|2(a''W_2,V_1)_{L^2}| \le \Vert a''\Vert_\infty^2 \Vert W_2\Vert_{L^2}^2+\Vert V_1\Vert^2_{L^2}\le \Vert a''\Vert_\infty^2 E(t)+\Vert V_1\Vert^2_{L^2}\\
&|4(b'_3W_2,V^0)_{L^2}| \le  2 \Vert b'_3 \Vert^2_\infty E(t) + 2\Vert V^0\Vert_{L^2}^2  \\
&|4(b'_1W_2,V_1)_{L^2}| \le 2\Vert b'_1 \Vert^2_\infty \Vert W_2\Vert_{L^2}^2+ 2\Vert V_1\Vert_{L^2}^2 \le 2 \Vert b'_1 \Vert^2_\infty E(t) + 2\Vert V_1\Vert_{L^2}^2  \\
&|4(b'_2W_2,V_2)_{L^2}| \le 2 \Vert b'_2 \Vert^2_\infty E(t) + 2\Vert V_2\Vert_{L^2}^2  \\
&|2(b''_3W_2,U^0)_{L^2}| \le \Vert b''_3\Vert_\infty^2 \Vert W_2\Vert_{L^2}^2+\Vert U^0\Vert^2_{L^2} \le \Vert b''_3\Vert_\infty^2 E(t)+\Vert U^0\Vert^2_{L^2}\\
&|2(b''_1W_2,U_1)_{L^2}| \le \Vert b''_1\Vert_\infty^2 E(t)+\Vert U_1\Vert^2_{L^2} \\
&|2(b''_2W_2,U_2)_{L^2}| \le \Vert b''_2\Vert_\infty^2 E(t)+\Vert U_2\Vert^2_{L^2} \\
&|2(W_2,f_{xx})_{L^2}|\le \Vert W_2\Vert_{L^2}^2+ \Vert f_{xx}\Vert_{L^2}^2\le E(t)+\Vert f_{xx}\Vert_{L^2}^2.\\
\end{align*}
Therefore,
 \begin{align}
 \label{est_E_W_1_3x3}
 \frac{dE(t)}{dt}\le &3\Vert V_1\Vert_{L^2}^2+2\Vert V^0\Vert_{L^2}^2+2\Vert V_2\Vert_{L^2}^2+\Vert U^0\Vert^2_{L^2}+ \Vert U_1\Vert_{L^2}^2+ \Vert U_2\Vert_{L^2}^2 + \Vert f_{xx}\Vert^2_{L^2}\\ 
 &+\max(\Vert a''\Vert_\infty ,M_2,(1+\Vert b_3\Vert_\infty),2\Vert b_2 \Vert_\infty, \Vert a''\Vert_\infty^2 ,2 \Vert b'_1 \Vert^2_\infty ,\nonumber\\ 
 & \qquad  \qquad 2 \Vert b'_2 \Vert^2_\infty ,2 \Vert b'_3 \Vert^2_\infty , \Vert b''_1\Vert_\infty^2,\Vert b''_2\Vert_\infty^2,\Vert b''_3\Vert_\infty^2 )E(t). \nonumber
 \end{align}
 
From estimates \eqref{est_U_2_3x3}, \eqref{est_V_2_3x3}, and \eqref{est_U_1_3x3}, we get
\begin{align}\label{est_E_W_1_intermidiate1_3x3}
2\Vert V^0\Vert_{L^2}^2+2\Vert V_2\Vert_{L^2}^2+ \Vert U^0\Vert_{L^2}^2+ \Vert U_1\Vert_{L^2}^2+ \Vert U_2\Vert_{L^2}^2 \le C_4 \biggl(\Vert g_0\Vert_{H^2}^2+\Vert g_1\Vert_{H^1}^2+\int_{0}^t\Vert f(s)\Vert_{H^1}^2 \, ds\biggl)
\end{align}
where $C_4>0$ depends linearly on $T$, $\Vert a \Vert_\infty$, $\Vert b'_1 \Vert^2_\infty$, $\Vert b'_2 \Vert^2_\infty$,  $\Vert b'_3 \Vert^2_\infty$ and exponentially on $T$, $\Vert b_2 \Vert_\infty$, $\Vert b_3 \Vert_\infty$, $\Vert b'_1 \Vert^2_\infty$, $M_2$, $\Vert a''\Vert_\infty$. It remains to estimate $3\Vert V_1\Vert_{L^2}^2$. Since $\partial_t V_1=W_2$ we can write
\begin{align}
 \label{est_E_W_1_intermidiate2_3x3}
3\Vert V_1\Vert^2_{L^2}&=3\biggl\Vert \int_0^t \partial_tV_1(s)\, ds+V_1(0)\biggr\Vert^2\le 6\biggl\Vert \int_0^t \partial_tV_1(s)\, ds\biggr\Vert^2_{L^2}+6\Vert V_1(0)\Vert_{L^2}^2\\
&\le 6\biggl(\int_0^t\Vert W_2(s)\Vert_{L^2}\, ds\biggr)^2+6\Vert V_1(0)\Vert_{L^2}^2
\le 6t\int_0^t\Vert W_2(s)\Vert^2_{L^2}\, ds+6\Vert V_1(0)\Vert_{L^2}^2 \nonumber\\
&\le 6T\int_0^t E(s)\, ds+6\Vert g_0\Vert_{H^2}^2. \nonumber
\end{align}
Combining \eqref{est_E_W_1_intermidiate1_3x3} with \eqref{est_E_W_1_intermidiate2_3x3}, we rewrite \eqref{est_E_W_1_3x3} as
 \begin{align}
 \label{est_E_W_2_3x3}
 \frac{dE(t)}{dt}\le& 6T\int_0^t E(s)\, ds+ \max(6,C_4) \biggl(\Vert g_0\Vert_{H^2}^2+\Vert g_1\Vert_{H^1}^2+\int_{0}^t\Vert f(s)\Vert_{H^1}^2 \, ds\biggl)   + \Vert f_{xx}\Vert^2_{L^2}\nonumber \\
 &+\max(\Vert a''\Vert_\infty ,M_2,(1+\Vert b_3\Vert_\infty),2\Vert b_2 \Vert_\infty, \Vert a''\Vert_\infty^2 ,2 \Vert b'_1 \Vert^2_\infty ,\nonumber\\ 
 & \qquad  \qquad 2 \Vert b'_2 \Vert^2_\infty ,2 \Vert b'_3 \Vert^2_\infty , \Vert b''_1\Vert_\infty^2,\Vert b''_2\Vert_\infty^2,\Vert b''_3\Vert_\infty^2 )E(t). \nonumber \nonumber
 \end{align}

By Lemma \ref{lem_ST_G_3x3}, we conclude that there exists a constant $C_5>0$, depending exponentially on $T$, $\Vert a''\Vert_\infty$, $M_2$, $\Vert b_2 \Vert_\infty$, $\Vert b_3 \Vert_\infty$, $\Vert a''\Vert_\infty^2$, $\Vert b'_1 \Vert^2_\infty$, $\Vert b'_2 \Vert^2_\infty$, $\Vert b'_3 \Vert^2_\infty$, $\Vert b''_1\Vert_\infty^2$, $\Vert b''_2\Vert_\infty^2$ and $\Vert b''_3\Vert_\infty^2$, such that 
\begin{align}
&\Vert W^0\Vert_{L^2}^2+\Vert W_2\Vert_{L^2}^2\le E(t)\nonumber\\
&\le C_5 \biggl( E(0) + \int_{0}^t   \max(6,C_4) \biggl(\Vert g_0\Vert_{H^2}^2+\Vert g_1\Vert_{H^1}^2+\int_{0}^s\Vert f(r)\Vert_{H^1}^2 \, dr\biggl)+ \Vert f_{xx}(s)\Vert^2_{L^2}  \,ds \biggr) \nonumber\\
&\le C_5\max(6,C_4)  \biggl( E(0) + \int_{0}^t\Vert f(s)\Vert_{H^1}^2 \, ds+  \int_{0}^t   \Vert f_{xx}(s)\Vert^2_{L^2}  \,ds +T(\Vert g_0\Vert_{H^2}^2+\Vert g_1\Vert_{H^1}^2)\biggr) \nonumber\\
&\le C_5\max(6,C_4)  \biggl( E(0) + \int_{0}^t\Vert f(s)\Vert_{H^2}^2 \, ds +T\left(\Vert g_0\Vert_{H^2}^2+\Vert g_1\Vert_{H^1}^2\right)\biggr) \nonumber\\
&\le C_5\max(6,C_4)  \biggl(\Vert W^0(0)\Vert_{L^2}^2 + \Vert a\Vert_\infty \Vert W_1(0)\Vert_{L^2}^2+\Vert W_2(0)\Vert_{L^2}^2+\nonumber\\
& \int_{0}^t\Vert f(s)\Vert_{H^2}^2 \, ds +T\left(\Vert g_0\Vert_{H^2}^2+\Vert g_1\Vert_{H^1}^2\right)\biggr) \nonumber\\
&\le C_5\max(6,C_4)  \biggl(\Vert g_0\Vert_{H^2}^2 +\Vert a\Vert_\infty \Vert g_0\Vert_{H^3}^2+\Vert g_1\Vert_{H^2}^2+ \int_{0}^t\Vert f(s)\Vert_{H^2}^2 \, ds +T\left(\Vert g_0\Vert_{H^2}^2+\Vert g_1\Vert_{H^1}^2\right)\biggr) \nonumber\\
&\le C_5\max(6,C_4,T,\Vert a\Vert_\infty )\biggl(\Vert g_0\Vert_{H^3}^2+\Vert g_1\Vert_{H^2}^2+\int_{0}^t \Vert f(s)\Vert_{H^2}^2 \, ds\biggr)\nonumber\\
&= C'_5\biggl(\Vert g_0\Vert_{H^3}^2+\Vert g_1\Vert_{H^2}^2+\int_{0}^t \Vert f(s)\Vert_{H^2}^2 \, ds\biggr)\nonumber.
\end{align}
Note that the constant $C'_5$ depends linearly on $T$, $\Vert a \Vert_\infty$, $\Vert b'_1 \Vert^2_\infty$,  $\Vert b'_2 \Vert^2_\infty$,  $\Vert b'_3 \Vert^2_\infty$ and exponentially on $T$, $\Vert a''\Vert_\infty$, $M_2$, $\Vert b_2 \Vert_\infty$, $\Vert b_3 \Vert_\infty$, $\Vert a''\Vert_\infty^2$, $\Vert b'_1 \Vert^2_\infty$, $\Vert b'_2 \Vert^2_\infty$, $\Vert b'_3 \Vert^2_\infty$, $\Vert b''_1\Vert_\infty^2$, $\Vert b''_2\Vert_\infty^2$ and $\Vert b''_3\Vert_\infty^2$.  Noting that $W^0= V_1$, we have that
\beq
\label{est_V_1_3x3}
\Vert V_1(t)\Vert^2_{L^2}\le C'_5\biggl(\Vert g_0\Vert_{H^3}^2+\Vert g_1\Vert_{H^2}^2+\int_{0}^t \Vert f(s)\Vert_{H^2}^2 \, ds\biggr). 
\eeq

\subsection{Sobolev estimates} Bringing everything together, we have proven that if $U$ is a solution of the Cauchy problem 
\[
\begin{split}
\partial_t U=&A \partial_x U+F,\\
U(0,x)&=(g_0,g'_0,g_1)^{T},
\end{split}
\]
then $\Vert U_1(t)\Vert^2_{L^2}$ is bounded by 
\[
C_3 \max(\Vert a\Vert_\infty,  C_{2} T\Vert b'_2 \Vert^2_\infty, C_{2} T\Vert b'_3 \Vert^2_\infty,1,2T\Vert b'_1 \Vert^2_\infty)\biggl(\Vert g_0\Vert_{H^2}^2+\Vert g_1\Vert_{H^1}^2+\int_{0}^t\Vert f(s)\Vert_{H^1}^2 \, ds\biggl)
\]
and 
\[
\Vert U^0(t)\Vert_{L^2}^2+\Vert U_{2}(t)\Vert_{L^2}^2\le C_{2}\biggl(\Vert g_0\Vert_{H^1}^2+\Vert g_1\Vert_{L^2}^2+\int_{0}^t\Vert f(s)\Vert_{L^2}^2\, ds\biggr).
\]
It follows that 
\[
\Vert U(t)\Vert_{L^2}^2\le A_0\biggl(\Vert g_0\Vert_{H^2}^2+\Vert g_1\Vert_{H^1}^2+\int_{0}^t\Vert f(s)\Vert_{H^1}^2\, ds\biggr)
\]
where $A_0$ depends linearly on $T$, $\Vert a\Vert_\infty$, $\Vert b'_1\Vert^2_\infty$,  $\Vert b'_2\Vert^2_\infty$, $\Vert b'_3\Vert^2_\infty$ and exponentially on $T$, $M_2$, $\Vert a''\Vert_\infty$, $\Vert b'_1 \Vert^2_\infty$, $\Vert b_2\Vert_\infty$, $\Vert b_3\Vert_\infty$. 

Passing now to the system in $V$ we have proven that 
\[
\begin{split}
\Vert V_1(t)\Vert_{L^2}^2&\le C'_5\biggl(\Vert g_0\Vert_{H^3}^2+\Vert g_1\Vert_{H^2}^2+\int_{0}^t \Vert f(s)\Vert_{H^2}^2 \, ds\biggr),\\
\Vert V^0(t)\Vert_{L^2}^2+\Vert V_{2}(t)\Vert_{L^2}^2&\le  C_3 \max(\Vert a\Vert_\infty,  C_{2} T\Vert b'_2 \Vert^2_\infty, C_{2} T\Vert b'_3 \Vert^2_\infty,1,2T\Vert b'_1 \Vert^2_\infty)\\  &\quad \times\biggl(\Vert g_0\Vert_{H^2}^2+\Vert g_1\Vert_{H^1}^2+\int_{0}^t\Vert f(s)\Vert_{H^1}^2 \, ds\biggl),\\
\end{split}
\]
and therefore there exists a constant $A_1>0$, depending linearly on $T$,  $\Vert a\Vert_\infty$, $\Vert b'_1 \Vert^2_\infty$, $\Vert b'_2 \Vert^2_\infty$, $\Vert b'_3 \Vert^2_\infty$ and  exponentially on $T$, $M_2$, $\Vert a''\Vert_\infty$, $\Vert a''\Vert_\infty^2$, $\Vert b_2 \Vert_\infty$, $\Vert b_3 \Vert_\infty$,   $\Vert b'_1 \Vert^2_\infty$, $\Vert b'_2 \Vert^2_\infty$, $\Vert b'_3 \Vert^2_\infty$, $\Vert b''_1\Vert_\infty^2$,   $\Vert b''_2\Vert_\infty^2$, $\Vert b''_3\Vert_\infty^2$,  such that
\[
\Vert U(t)\Vert_{H^1}^2\le A_1\biggl(\Vert g_0\Vert_{H^3}^2+\Vert g_1\Vert_{H^2}^2+\int_{0}^t\Vert f(s)\Vert_{H^2}^2\, ds\biggr).
\]
This immediately gives the estimates
\[
\Vert u(t)\Vert_{H^{k+1}}^2\le A_k\biggl(\Vert g_0\Vert_{H^{k+2}}^2+\Vert g_1\Vert_{H^{k+1}}^2+\int_0^t \Vert f(s)\Vert_{H^{k+1}}^2\, ds\biggl),
\]
for all $t\in[0,T]$ and $k=-1,0,1$, where $A_k$ depends  linearly on $T^{\frac{(k+1)(2-k)}{2}}$,  $\Vert a\Vert_\infty$, $\Vert b'_1 \Vert^{(k+1)(2-k)}_\infty$, $\Vert b'_2 \Vert^{(k+1)(2-k)}_\infty$, $\Vert b'_3 \Vert^{(k+1)(2-k)}_\infty$ and  exponentially on $T$, $M_2$, $\Vert a''\Vert_\infty$, $\Vert a''\Vert_\infty^{k(k+1)}$,  $\Vert b'_1 \Vert^2_\infty$, $\Vert b''_1\Vert_\infty^{k(k+1)}$, $\Vert b_2 \Vert_\infty$, $\Vert b'_2 \Vert^{k(k+1)}_\infty$,  $\Vert b''_2\Vert_\infty^{k(k+1)}$, $\Vert b_3 \Vert_\infty$, $\Vert b'_3 \Vert^{k(k+1)}_\infty$,  $\Vert b''_3\Vert_\infty^{k(k+1)}$.

\subsection{Conclusion} We have obtained $H^k$-Sobolev well-posedness for the Cauchy problem \eqref{CP_1_ex} for $k=0,1,2$, provided that $a\ge 0$ and $a, b_1, b_2, b_3\in B^{\infty}(\R)$. We can obtain existence of the solution via a standard perturbation argument on the strictly hyperbolic case (see \cite{ST} and the proof of Theorem \ref{theo_main}) and from the above estimates we can obtain uniqueness. This argument can be iterated to obtain Sobolev estimates for every order $k$. The iteration will involve higher-order derivatives of the coefficients. More precisely, if we want to estimate the $H^k$-norm of $U(t)$ then we will derive the coefficients $a_j$ up to order $k+1$. We therefore have the following proposition.
 \begin{proposition}
 \label{prop_est_Sob_k}
 Assume that $a\ge 0$ and $a, b_1, b_2, b_3\in B^{\infty}(\R)$. Then, for all $k\in\N$ there exists a constant $A_k$ depending on $T$, $M_2$ and the $L^\infty$-norms of the derivatives of the coefficients up to order $k+1$ such that
 \beq
\label{est_U_H_k}
\Vert U(t)\Vert_{H^k}^2\le A_k\biggl(\Vert g_0\Vert_{H^{k+2}}^2+\Vert g_1\Vert_{H^{k+1}}^2+\int_0^t \Vert f(s)\Vert_{H^{k+1}}^2\, ds\biggl),
\eeq
for all $t\in[0,T]$.
 \end{proposition}

\subsection{Existence and uniqueness result}
We now prove that the Cauchy problem \eqref{CP_1_ex}
\begin{equation*}
\begin{split}
\partial_t^2u-a(x)\partial^2_{x} u+ b_1(x)\partial_x u+ b_2(x)\partial_t u + b_3(x) u&=f(t,x),\quad t\in[0,T],\, x\in\R,\\
u(0,x)&=g_0,\\
\partial_tu(0,x)&=g_1,
\end{split}
\end{equation*}
is well-posed in every Sobolev space and hence in $C^\infty(\R)$. We will use the estimate \eqref{est_U_H_k} which we  re-write in terms of $u$ as
\beq
\label{est_u_H_k}
\Vert u(t)\Vert_{H^{k+1}}^2\le A_k\biggl(\Vert g_0\Vert_{H^{k+2}}^2+\Vert g_1\Vert_{H^{k+1}}^2+\int_0^t \Vert f(s)\Vert_{H^{k+1}}^2\, ds\biggl).
\eeq
\begin{theorem}
\label{theo_main}
Let  $a\ge 0$ and $a, b_1, b_2, b_3\in B^{\infty}(\R)$. Assume that $g_0,g_1 \in C_c^{\infty} (\R)$ and $f\in C([0,T], C_c^{\infty}(\R))$. Then, the Cauchy problem \eqref{CP_1_ex} is well-posed in every Sobolev space $H^k$, with $k\in\N$, and 
for all $k\in\N$ there exists a constant $C_k>0$ depending on  $T$, $M_2$ and the $L^\infty$-norms of the derivatives of the coefficients up to order $k+1$ such that
 \beq
\label{est_u_H_k_1}
\Vert u(t)\Vert_{H^{k+1}}^2\le C_k\biggl(\Vert g_0\Vert_{H^{k+2}}^2+\Vert g_1\Vert_{H^{k+1}}^2+\int_0^t \Vert f(s)\Vert_{H^{k+1}}^2\, ds\biggl),
\eeq
for all $t\in[0,T]$.
\end{theorem}

\begin{proof}

\leavevmode
\begin{itemize}
 \item[(i)] {\bf Existence}. Assume that $f\in {C}([0,T], {C}_c^{\infty} (\R))$ and let 
 \[
 P(u)=\partial_t^2u-a(x)\partial^2_{x} u+ b_1(x)\partial_x u+ b_2(x)\partial_t u + b_3(x) u.
 \]
  The strictly hyperbolic Cauchy problem 
 \[
 \begin{split}
 P_\delta(u)&=\partial_t^2u-(a(x)+\delta)\partial^2_{x} u+ b_1(x)\partial_x u+ b_2(x)\partial_t u + b_3(x) u=f,\\
 u(0,x)&=g_0(x)\in{C}_c^{\infty}(\R),\\
\partial_t u(0,x)&=g_1(x)\in{C}_c^{\infty}(\R),
\end{split}
 \]
 has a unique solution $(u_\delta)_\delta$ defined via  $(U_\delta)_\delta$, the corresponding vector. Since, we can choose the constant 
 \[A=A(T,  M_2, \Vert a+\delta\Vert_{\infty}, \Vert a''\Vert_{\infty}, \Vert b_2\Vert_\infty, \Vert b_3\Vert_\infty, \Vert b'_1\Vert_\infty,  \Vert b'_2\Vert_\infty , \Vert b'_3\Vert_\infty)>0\]
 independent of $\delta\in(0,1)$, we have that when $g_0,g_1(x) \in{C}_c^{\infty} (\R)$, the net 
 \[
 U_\delta=(U^0_{\delta}, U_{1,\delta},U_{2,\delta}) 
 \]
 is bounded in $(L^2(\R))^{3}$. Therefore there exists a convergent subsequence in $(L^2(\R))^{3}$ with limit $U\in (L^2(\R))^{3}$ that solves the system 
\[
\begin{split}
\partial_t U=&A \partial_x U+F,\\
U(0,x)&=(g_0,g'_0,g_1)^{T},
\end{split}
\]
in the sense of distributions. The arguments that we apply are similar as in the homogeneous case, in the proof of Theorem 4.12  in \cite{G:21}.

 \item[(ii)] {\bf Uniqueness}. The uniqueness of the solution $u$ follows immediately from the estimate \eqref{est_u_H_k_1}. 
 \end{itemize}

\end{proof}
As a straightforward consequence we get the following result of $C^\infty$ well-posedness. It extends the well-posedness result obtained in \cite{ST} to any space dimension and is consistent with Oleinik's result \cite{O70}. As we have bounded coefficients, we obtain global well-posedness instead of local well-posedness.
\begin{corollary}
\label{cor_C_infty}
Let $a\ge 0$, $a, b_1, b_2, b_3\in B^{\infty}(\R)$ and let $f\in C([0,T], C^\infty_c(\R))$. Then the Cauchy problem \eqref{CP_1_ex} is $C^\infty(\R)$ well-posed, i.e., given $g_0,g_1\in{C}^\infty_c(\R)$ there exists a unique solution $C^2([0,1], {C}^\infty(\R))$ of 
\[
\begin{split}
\partial_t^2u-a(x)\partial^2_{x} u+ b_1(x)\partial_x u+ b_2(x)\partial_t u + b_3(x) u&=f(t,x),\quad t\in[0,T],\, x\in\R,\\
u(0,x)&=g_0,\\
\partial_tu(0,x)&=g_1.
\end{split}
\]
Moreover, the estimate \eqref{est_u_H_k_1} holds for every $k\in\N$.
\end{corollary}
Note that one can remove the assumption of compact support on $f$ and the initial data, by the finite speed of propagation. 

\section{The inhomogeneous wave equation with space-dependent singular coefficients}
When dealing with singular coefficients, one can encounter equations that may not have a meaningful classical distributional solution. This is due to the well-known problem that in general it is not possible to multiply two arbitrary distributions. To handle this, the notion of a very weak solution has been introduced in \cite{GR:14}. In \cite{GR:14}, the authors were looking for solutions modelled on Gevrey spaces. However, in this paper we will instead prove Sobolev well-posedness with loss of derivatives. This motivates the introduction of \emph{very weak solutions of Sobolev type}.
 
\subsection{Very weak solutions of Sobolev type}

In the sequel, let $\varphi$ be a mollifier ($\varphi\in C^\infty_c(\R^n)$, $\varphi\ge 0$ with $\int\varphi=1$) and let $\omega(\eps)$  a positive net converging to $0$ as $\eps\to 0$. Let $\varphi_{\omega(\eps)}(x)=\omega(\eps)^{-n}\varphi(x/\omega(\eps))$. $K\Subset\R^n$ stands for $K$ is a compact subset of $\R^n$. We are now ready to introduce the concepts of  $H^k$-moderate and $H^k$-negligible nets.

\begin{definition}
\label{def_mod_neg}
\leavevmode
\begin{itemize}
\item[(i)] A net $(v_\eps)_\eps\in H^k(\R^n)^{(0,1]}$ is $H^k$-moderate if there exist $N\in\N$ and $c>0$ such that
\beq
\label{mod_C_inf}
\| v_\eps(x)\|_{H^k(\R^n)} \le c\eps^{-N},
\eeq
uniformly in  $\eps\in(0,1]$.
\item[(ii)] A net $(v_\eps)_\eps\in H^k(\R^n)^{(0,1]}$  is $H^k$-negligible if for all $q\in\N$ there exists $c>0$ such that
\beq
\label{neg_C_inf}
\| v_\eps(x)\|_{H^k(\R^n)} \le c\eps^{q},
\eeq
uniformly in  $\eps\in(0,1]$.
\end{itemize}
\end{definition}

Since we will only be considering nets of $H^k$ functions, we can simply use the expressions \emph{moderate net} and \emph{negligible net} and drop the $H^k$-suffix. Note that Biagioni and Oberguggenberger introduced in \cite{BO:92} spaces  of generalised functions generated by nets in $H^\infty(\R^n)^{(0,1]}$. As a consequence, their notion of moderateness and negligibility involves derivatives of any order. 

Analogously, by considering $C^\infty([0,T]; H^k(\R^n))^{(0,1]}$ instead of  $H^k(\R^n)^{(0,1]}$, we formulate the following definition. 

\begin{definition}
\label{def_mod_neg_u}
\leavevmode
\begin{itemize}
\item[(i)] A net $(v_\eps)_\eps\in C^\infty([0,T]; H^k(\R^n))^{(0,1]}$ is moderate if for all $l\in\N$  there exist $N\in\N$ and $c>0$ such that
\beq
\label{mod_C_inf_u}
\|\partial_t^l v_\eps(t,\cdot)\|_{H^k(\R^n)} \le c\eps^{-N},
\eeq
uniformly in $t\in[0,T]$ and $\eps\in(0,1]$.
\item[(ii)] A net $(v_\eps)_\eps\in  C^\infty([0,T]; H^k(\R^n))^{(0,1]}$  is negligible if for all  $l\in\N$ and $q\in\N$ there exists $c>0$ such that
\beq
\label{neg_C_inf_u}
\|\partial_t^l v_\eps(t,\cdot)\|_{H^k(\R^n)} \le c\eps^{q},
\eeq
uniformly in $t\in[0,T]$ and $\eps\in(0,1]$.
\end{itemize}
\end{definition}


In order to introduce the notion of a very weak solution for the Cauchy problem \eqref{CP_1_intro},
\[
\begin{split}
\partial_t^2u-a(x)\partial^2_{x} u+ b_1(x)\partial_x u+ b_2(x)\partial_t u + b_3(x) u&=f(t,x),\quad t\in[0,T],\, x\in\R,\\
u(0,x)&=g_0,\\
\partial_tu(0,x)&=g_1,
\end{split}
\]
where $a\ge 0$ and in general all coefficients, initial data and right-hand side are compactly supported distributions, we need some preliminary work on how to regularise our equation. We will provide estimates in terms of $L^\infty$- as well as $L^2$-norm and we will focus on coefficients and initial data in $L^\infty(\R^n)$, $\E'(\R^n)$ and $B^{\infty}(\R^n)$. 
\begin{proposition}\label{lemma_mollification}
Let $\varphi$ be a mollifier ($\varphi\in C^\infty_c(\R^n)$, $\varphi\ge 0$ with $\int\varphi=1$) and $\omega(\eps)$  a positive function converging to $0$ as $\eps\to 0$. Let $\varphi_{\omega(\eps)}(x)=\omega(\eps)^{-n}\varphi(x/\omega(\eps))$. 
\begin{itemize}
\item[(i)] If $v \in L^\infty(\R^n)$, then  $\forall \eps \in (0,1)$, $v \ast \varphi_{\omega(\eps)}\in B^\infty(\R^n)$ and
\[
\forall \alpha \in \N^n, \,\Vert \partial^{\alpha}(v \ast \varphi_{\omega(\eps)}) \Vert_\infty \le \omega(\eps)^{-|\alpha|} \Vert v\Vert_\infty \Vert \partial^{\alpha}\varphi \Vert_{L^1}.
\]
\item[(ii)] If $v \in \E'(\R^n)$, then  $\forall \eps \in (0,1)$, $v \ast \varphi_{\omega(\eps)}\in B^\infty(\R^n)$ and
\[
\forall \alpha \in \N^n, \,\Vert \partial^{\alpha}(v \ast \varphi_{\omega(\eps)}) \Vert_\infty \le  \omega(\eps)^{-|\alpha|-m} \sum_{|\beta|\le m } \Vert  g_{\beta} \Vert_\infty \Vert \partial^{\alpha+ \beta}\varphi \Vert_{L^1},
\]
where $m \in \N$ and $g_\beta \in C_c(\R^n)$ come from the structure of $v$.
\item[(iii)] If $v \in B^{\infty}(\R^n)$, then  $\forall \eps \in (0,1)$, $v \ast \varphi_{\omega(\eps)}\in B^\infty(\R^n)$ and
\[
\forall \alpha \in \N^n, \,\Vert \partial^{\alpha}(v \ast \varphi_{\omega(\eps)}) \Vert_\infty \le  \Vert \partial^{\alpha}v \Vert_\infty .
\]
\end{itemize}
\end{proposition}

\begin{proof}
\begin{itemize}
\item[(i)] By the properties of convolution, 
\begin{align*}
\Vert \partial^{\alpha}(v \ast \varphi_{\omega(\eps)}) \Vert_\infty &= \Vert v \ast (\partial^{\alpha}\varphi_{\omega(\eps)} )\Vert_\infty= \omega(\eps)^{-|\alpha|-n}\Vert v \ast (\partial^{\alpha}\varphi)\left(\frac{\cdot}{\omega(\eps)}\right) \Vert_\infty \\
&=  \omega(\eps)^{-|\alpha|-n}\Vert \int_{\R^n} v(\cdot-\xi)  (\partial^{\alpha}\varphi)\left(\frac{\xi}{\omega(\eps)}\right)d\xi \Vert_\infty \\
&=  \omega(\eps)^{-|\alpha|}\Vert \int_{\R^n} v(\cdot-\omega(\eps)z)  (\partial^{\alpha}\varphi)\left(z\right)dz \Vert_\infty \\
&\le \omega(\eps)^{-|\alpha|} \Vert v\Vert_\infty \Vert \partial^{\alpha}\varphi \Vert_{L^1}.
\end{align*}
\item[(ii)] By the structure theorem of compactly supported distributions, there exists $m \in \N$ and $g_\beta \in C_c(\R^n)$ such that
\begin{align*}
\Vert \partial^{\alpha}(v \ast \varphi_{\omega(\eps)}) \Vert_\infty &= \Vert ( \sum_{|\beta|\le m } \partial^{\beta} g_{\beta}  ) \ast (\partial^{\alpha}\varphi_{\omega(\eps)} )\Vert_\infty=  \Vert  \sum_{|\beta|\le m }( \partial^{\beta} g_{\beta}   \ast \partial^{\alpha}\varphi_{\omega(\eps)} ) \Vert_\infty\\
&= \Vert  \sum_{|\beta|\le m }(  g_{\beta}   \ast \partial^{\alpha+ \beta}\varphi_{\omega(\eps)} ) \Vert_\infty \le\sum_{|\beta|\le m } \Vert    g_{\beta}   \ast \partial^{\alpha+ \beta}\varphi_{\omega(\eps)} \Vert_\infty \\ &\le \omega(\eps)^{-|\alpha|-m-n}  \sum_{|\beta|\le m } \Vert  g_{\beta}   \ast (\partial^{\alpha+ \beta}\varphi)\left(\frac{\cdot}{\omega(\eps)}\right) \Vert_\infty   \\
&\le  \omega(\eps)^{-|\alpha|-m} \sum_{|\beta|\le m } \Vert  g_{\beta} \Vert_\infty \Vert \partial^{\alpha+ \beta}\varphi \Vert_{L^1}.
\end{align*} 
\item[(iii)] Putting the derivatives on $v$, we get
\begin{align*}
\Vert \partial^{\alpha}(v \ast \varphi_{\omega(\eps)}) \Vert_\infty &= \Vert (\partial^{\alpha}v) \ast \varphi_{\omega(\eps)} \Vert_\infty \le \Vert \partial^{\alpha}v \Vert_\infty \Vert\varphi_{\omega(\eps)} \Vert_{L^1}\\
&=\Vert \partial^{\alpha}v \Vert_\infty .
\end{align*}
\end{itemize}
\end{proof}
By employing Theorem 2.7 in \cite{BO:92} we get the following result of $L^2$-moderateness, where $\varphi$ and $\omega(\eps)$ are defined as above.
\begin{proposition}\label{prop_rhs_mod}
Let $w \in \mathcal{E}'(\R^n)$. Then, there exists $N \in \N$ and for  $\beta \in \N^n$  there exists $C>0$ such that,
\[
\Vert \partial^{\beta} (w \ast \varphi_{\omega(\eps)}(x) )\Vert_{L^2} \le C \omega(\eps)^{-N-|\beta|},
\]
where $N$ depends on the structure of $w$.
\end{proposition}
\begin{proof}
By the structure theorem of compactly supported distributions, there exists $m \in \N$ and $w_\alpha \in C_c(\R^n)$ such that $ w=\sum_{|\alpha|\le m } \partial^{\alpha} w_{\alpha}$. For  $\beta \in \N^n$, by Young's inequality we get
\[
\Vert \partial^{\beta} (w \ast \varphi_{\omega(\eps)}(x) )\Vert_{L^2}\le \sum_{|\alpha|\le m } \Vert w_{\alpha} \Vert_{L^2}\Vert \partial^{\alpha+\beta} \varphi_{\omega(\eps)} \Vert_{L^1}=\sum_{|\alpha|\le m } \Vert w_{\alpha} \Vert_{L^2} \Vert \partial^{\alpha+\beta} \varphi\Vert_{L^1}\omega(\eps)^{-|\alpha-\beta|} .
\]
\end{proof}
Note that if $f\in C([0,T], \mathcal{E}'(\R^n)$ then the net
\[
(f(t,\cdot)\ast \varphi_{\omega(\eps)})(x)
\]
fulfils moderate estimates with respect to $x$ which are uniform with respect to $t\in[0,T]$. We now have all the needed background to go back to the Cauchy problem \eqref{CP_1_intro} and formulate the appropriate notion of very weak solution. We work under the assumptions that all the coefficients are distributions with compact support with $a$ positive distribution, $f\in C([0,T], \E'(\R))$ and $g_0,g_1\in \E'(\R)$.  As a first step, by convolution with a mollifier as in Proposition \ref{lemma_mollification}, we regularise the Cauchy problem \eqref{CP_1_intro} and get 
\beq
\label{CP_1_ex_reg}
\begin{split}
\partial_t^2u_{\eps}-a_{\eps}(x)\partial^2_{x} u_{\eps}+ b_{1,\eps}(x)\partial_x u_{\eps}+ b_{2,\eps}(x)\partial_t u_{\eps} + b_{3,\eps}(x) u_{\eps}&=f_{\eps}(t,x),\\
u_{\eps}(0,x)&=g_{0,\eps}(x),\\
\partial_tu_{\eps}(0,x)&=g_{1,\eps}(x),
\end{split}
\eeq
where $t\in[0,T]$, $x\in\R$, $a_{\eps}=a\ast \varphi_{\omega(\eps)}$, $b_{1,\eps}=b_{1}\ast \varphi_{\omega(\eps)}$, $b_{2,\eps}=b_{2}\ast \varphi_{\omega(\eps)}$, $b_{3,\eps}=b_{3}\ast \varphi_{\omega(\eps)}$, $ f_{\eps}(t,x)= f(t,\cdot)\ast \varphi_{\omega(\eps)}(x)$, $g_{0,\eps}=g_0\ast \varphi_{\omega(\eps)}$ and $g_{1,\eps}=g_1\ast \varphi_{\omega(\eps)}$.
We also assume that the regularised nets are real-valued (easily obtained with our choice of mollifiers). 

We also impose the Levi condition
\beq	 \label{Levi_mollified_Case1}
b_{1,\eps}(x)^2 \le M_{2,\eps} a_{\eps}(x).
\eeq
This will guarantee that the regularised Cauchy problem above has a net of smooth solutions $(u_\eps)_\eps$.

\begin{definition}
The net of solutions $u_\varepsilon(t,x)$ is a {\bf very weak solution of Sobolev order k} of the Cauchy problem \eqref{CP_1_intro}
if there exist $N \in \N$ and $c>0$  such that
\[
\| u_\varepsilon(t,\cdot)\|_{H^k(\R)}\le c  \varepsilon^{-N},
\]
for all $t\in[0,T]$ and $\varepsilon \in(0,1]$.
\end{definition}
In other words, the Cauchy problem \eqref{CP_1_intro} has a very weak solution of order $k$ if there exists moderate regularisations of coefficients, right-hand side and initial data such that the corresponding net of solutions $(u_\eps)_\eps$ is Sobolev moderate of order $k$. This definition is in line with the one introduced in \cite{GR:14} but moderateness is measured in terms of Sobolev norms rather than $C^\infty$-seminorms. 
\begin{remark}
Uniqueness of very weak solutions is formulated, as for algebras of generalised functions, in terms of negligible nets. Namely, we say that the Cauchy Problem \eqref{CP_1_intro} is \emph{very weakly well-posed} if a very weak solution exists and it is unique modulo negligible nets. We will discuss this more in details later in the paper.
\end{remark}

\subsection{Existence of a very weak solution}

In this subsection we want to understand which regularisations entail the existence of a very weak solution.

By using the transformation,
\[
U_{\eps}=(U_{\eps}^0,U_{1, \eps},U_{2,\eps})^T=(u_{\eps},\partial_{x}u_{\eps},\partial_tu_{\eps})^{T},
\]
we can rewrite the above Cauchy problem as
\[
\begin{split}
\partial_t U_{\eps}&=A_{\eps} \partial_x U_{\eps}+ B_{\eps} U_{\eps}+F_{\eps},\\
U_{\eps}(0,x)&=(g_{0,\eps},g_{0,\eps}',g_{1,\eps})^{T},
\end{split}
\]
where
\[
A=\left(
	\begin{array}{ccc}
	0& 0 & 0\\
	0& 0 & 1\\
    0& a_{\eps} & 0 
	\end{array}
	\right),
	\quad
B=\left(
	\begin{array}{ccc}
	0& 0 & 1\\
	0& 0 & 0\\
    -b_{3,\eps} &-b_{1,\eps} & -b_{2,\eps} 
	\end{array}
	\right) 
	\text{ and }	\quad	
F=\left(
	\begin{array}{c}
	0\\
	0 \\
	f_{\eps}
        \end{array}
	\right).
	\]

The symmetriser of the  matrix $A$ is
\[
Q_{\eps}=\left(
	\begin{array}{ccc}
	1&0&0 \\
	0&a_{\eps}& 0\\
    0&0 & 1 
	\end{array}
	\right),
\]
and the Energy is defined as follows:
\[
E_{\eps}(t)=(Q_{\eps}U_{\eps},U_{\eps})_{L^2}=\Vert U_{\eps}^0\Vert^2_{L^2}+(a_{\eps}U_{1,\eps},U_{1,\eps})_{L^2}+\Vert U_{2,\eps}\Vert^2_{L^2}.
\]
Note that since both $a$ and $\varphi_\eps$ are non-negative we have that $a_{\eps}\ge 0$ and therefore the bound from below
\[
\Vert U_{\eps}^0\Vert^2_{L^2}+\Vert U_{2,\eps}\Vert^2_{L^2}\le E_{\eps}(t)
\]
holds. By arguing as in Section \ref{sec_Oleinik} we arrive at
\beq
\label{est_en_eps}
\frac{dE_{\eps}(t)}{dt}=-((Q_{\eps}A_{\eps})'U_{\eps},U_{\eps})_{L^2}+((Q_{\eps}B_{\eps}+B_{\eps}^*Q_{\eps})U_{\eps}, U_{\eps})_{L^2}+2(Q_{\eps}U_{\eps},F_{\eps})_{L^2}.
\eeq
Estimating the right-hand side of \eqref{est_en_eps} requires Glaeser's inequality. However, since we are working with nets, we have that 
\[
|a'_\eps(x)|^2\le 2M_{1,\eps}a_\eps(x),
\]
where 
\[
M_{1,\eps}=\Vert a''_\eps\Vert_{L^\infty}.
\]
Hence, 
\begin{align}\label{estimate1_3x3_Case1}
((Q_{\eps}A_{\eps})'U_{\eps},U_{\eps})_{L^2}&=2(a_{\eps}'U_{1,\eps},U_{2,\eps})_{L^2}\le 2\Vert a_{\eps}'U_{1,\eps}\Vert_{L^2}\Vert U_{2,\eps}\Vert_{L^2}\\
&\le \Vert a_{\eps}'U_{1,\eps}\Vert_{L^2}^2+\Vert U_{2,\eps}\Vert^2_{L^2}\le 2M_{1,\eps}(a_{\eps}U_{1,\eps},U_{1,\eps})_{L^2}+\Vert U_{2,\eps}\Vert^2_{L^2}\nonumber \\
&\le \max(2M_{1,\eps},1)E_{\eps}(t). \nonumber
\end{align}
Making use of the Levi condition  \eqref{Levi_mollified_Case1}, 
\[
b_{1,\eps}(x)^2 \le M_{2,\eps} a_{\eps}(x),
\]
arguing as in Section \ref{sec_Oleinik} we obtain
\begin{align} \label{estimate2_3x3_Case1}
&((Q_{\eps}B_{\eps}+B_{\eps}^*Q_{\eps})U_{\eps}, U_{\eps})_{L^2}\\
& \le \max(M_{2,\eps},1+\Vert b_{3,\eps}\Vert_\infty, 1+2\Vert b_{2,\eps}\Vert_\infty,2+2\Vert b_{2,\eps}\Vert_\infty+\Vert b_{3,\eps}\Vert_\infty)E_{\eps}(t). \nonumber
\end{align}
Lastly,
\beq \label{estimate3_3x3_Case1}
2(Q_{\eps}U_{\eps},F_{\eps})_{L^2}=2(U_{2,\eps},f_{\eps})_{L^2}\le 2\Vert U_{2,\eps}\Vert_{L^2}\Vert f_{\eps}\Vert_{L^2} \le\Vert U_{2,\eps}\Vert_{L^2}^2+\Vert f_{\eps}\Vert_{L^2}^2 \le E_{\eps}(t)+\Vert f_{\eps}\Vert_{L^2}.
\eeq

Therefore from \eqref{estimate1_3x3_Case1}, \eqref{estimate2_3x3_Case1} and \eqref{estimate3_3x3_Case1}, we obtain
\[
\frac{dE_{\eps}}{dt}\le c_{\eps}E_{\eps}(t)+\Vert f_{\eps}\Vert_{L^2},
\]
where
\[
c_{\eps}=\max(2M_{1,\eps}+1+M_{2,\eps},  2M_{1,\eps}+3+2\Vert b_{2,\eps}\Vert_\infty+\Vert b_{3,\eps}\Vert_\infty,2+M_{2,\eps},4+2\Vert b_{2,\eps}\Vert_\infty+\Vert b_{3,\eps}\Vert_\infty).
\]

Applying Gr\"onwall's lemma and using the bound from below for the energy, we obtain the estimate following estimate
\beq \label{est_U_2_3x3_Case1}
\Vert u_{\eps}(t)\Vert_{L^2}^2\le C_{\eps}\biggl(\Vert g_{0,\eps}\Vert_{H^1}^2+\Vert g_{1,\eps}\Vert_{L^2}^2+\int_{0}^t\Vert f_{\eps}(s)\Vert_{L^2}^2\, ds\biggr),
\eeq
 where the constant $C_{\eps}$ depends linearly on $\Vert a_{\eps}\Vert_\infty$ and exponentially on $T$, $M_{1,\eps}$, $M_{2,\eps}$,  $\Vert b_{2,\eps}\Vert_\infty$ and $\Vert b_{3,\eps}\Vert_\infty$. In particular,
 \beq
\label{C_2_3x3_Case 1}
\begin{split}
C_{\eps}&= {\rm e}^{c_{\eps}T}\max(\Vert a_{\eps}\Vert_{\infty},1)\\
= &{\rm e}^{\max(2M_{1,\eps}+1+M_{2,\eps},  2M_{1,\eps}+3+2\Vert b_{2,\eps}\Vert_\infty+\Vert b_{3,\eps}\Vert_\infty,2+M_{2,\eps},4+2\Vert b_{2,\eps}\Vert_\infty+\Vert b_{3,\eps}\Vert_\infty) T}\max(\Vert a_\eps\Vert_{\infty},1).
\end{split}
\eeq

By looking at the estimate above it is natural to formulate three different cases for the coefficients.
\begin{itemize}
\item {\bf Case 1:} the coefficients $a, b_1, b_2, b_3\in L^{\infty}(\R)$ and $a(x)\ge 0$ for all $x\in\R$.
\item {\bf Case 2:} the coefficients $a, b_1, b_2, b_3\in \E'(\R)$ and $a\ge 0$.
\item {\bf Case 3:} the coefficients $a, b_1, b_2, b_3\in B^{\infty}(\R)$ and $a(x)\ge 0$ for all $x\in\R$.
\end{itemize}

The analysis of these cases will require the following parametrised version of Glaeser's inequality, that for the sake of completeness, we formulate in $\R^n$. We recall that in $\R^n$, Glaeser's inequality is formulated as follows: 
If $a\in C^2(\R^n)$, $a(x)\ge 0$ for all $x\in\R^n$ and 
\[
\sum_{i=1}^n\Vert \partial^2_{x_i}a\Vert_{L^\infty}\le M,
\]
for some constant $M_1>0$. Then,  
\[
|\partial_{x_i}a(x)|^2\le 2M_1 a(x),
\]
for all $i=1,\dots,n$ and $x\in\R^n$. 

Hence, we immediately obtained its parametrised version for $a_\eps=a\ast\varphi_{\omega(\eps)}$, where $\varphi$ be a mollifier ($\varphi\in C^\infty_c(\R^n)$, $\varphi\ge 0$ with $\int\varphi=1$) and $\omega(\eps)$  a positive net converging to $0$ as $\eps\to 0$.
 
\begin{proposition} 
\label{prop_Glaeser_mollified}
\leavevmode
\begin{itemize}
\item[(i)] If $a \in L^\infty(\R^n)$ and $a\ge 0$ then   
\[
|\partial_{x_i}a_{\eps}(x)|^2\le 2M_{1,\eps} a_{\eps}(x),
\]
for all $i=1,\dots,n$, $x\in\R^n$ and $\forall \eps \in (0,1)$, where 
$$M_{1,\eps}=\omega(\eps)^{-2} \Vert a\Vert_\infty \sum_{i=1}^n \Vert \partial^2_{x_i}\varphi \Vert_{L^1}.$$
\item[(ii)]
If $a \in \E'(\R^n)$ and $a\ge 0$ then
\[
|\partial_{x_i}a_{\eps}(x)|^2\le 2M_{1,\eps} a_{\eps}(x),
\]
for all $i=1,\dots,n$, $x\in\R^n$ and $\forall \eps \in (0,1)$, where 
$$M_{1,\eps}=\omega(\eps)^{-2-m} \sum_{|\beta|\le m } \Vert  g_{\beta} \Vert_\infty \sum_{|\alpha| = 2} \Vert \partial^{\alpha+ \beta}\varphi \Vert_{L^1}$$
and $m \in \N$ and $g_\beta \in C_c(\R^n)$ come from the structure of $a$. 
\item[(iii)]
If $a \in  B^{\infty}(\R^n)$ and $a\ge 0$ then  
\[
|\partial_{x_i}a_{\eps}(x)|^2\le 2M_1 a_{\eps}(x),
\]
for all $i=1,\dots,n$, $x\in\R^n$ and $\forall \eps \in (0,1)$, where $M_1= \sum_{i=1}^n \Vert \partial^2_{x_i} a\Vert_\infty$.
\end{itemize}
\end{proposition}

In the rest of this section we will analyse the three different cases above and prove the existence of a very weak solution. In the sequel we say that $\omega(\eps)^{-1}$ is a scale of logarithmic type if $\omega(\eps)^{-1}\prec\ln(\eps^{-1})$.

\subsection{Case 1}
We assume that our coefficients $a, b_1, b_2, b_3\in L^{\infty}(\R)$ and $a(x)\ge 0$ for all $x\in\R$. From Proposition \ref{prop_Glaeser_mollified} (i), $M_{1,\eps}=\omega(\eps)^{-2} \Vert a\Vert_\infty \Vert \partial^2\varphi \Vert_{L^1}$ and $\Vert a_{\eps}\Vert_\infty\le  \Vert a\Vert_\infty$, $\Vert b_{2,\eps}\Vert_\infty\le  \Vert b_{2}\Vert_\infty$, $\Vert b_{3,\eps}\Vert_\infty\le  \Vert b_{3}\Vert_\infty$ from Proposition \ref{lemma_mollification} (i). The above calculations in combination with \eqref{C_2_3x3_Case 1}, \eqref{est_U_2_3x3_Case1} and Proposition \ref{prop_rhs_mod} show that we have proven the following theorem.
\begin{theorem}
\label{theo_main_case_1}
Let us consider the Cauchy problem
\beq
\label{CP_1_ex_case1}
\begin{split}
\partial_t^2u-a(x)\partial^2_{x} u+ b_1(x)\partial_x u+ b_2(x)\partial_t u + b_3(x) u&=f(t,x),\\
u(0,x)&=g_0,\\
\partial_tu(0,x)&=g_1,
\end{split}
\eeq
where $t\in[0,T]$, $x\in\R$.
Assume the following set of hypotheses (Case 1):
\begin{itemize}
\item[(i)] the coefficients $a, b_1, b_2, b_3\in L^{\infty}(\R)$ and $a(x)\ge 0$ for all $x\in\R$,
\item[(ii)] $f\in C([0,T],\mathcal{E}'(\R))$,
\item[(iii)] $g_0,g_1\in  \mathcal{E}'(\R)$.
\end{itemize}
If we regularise 
\[
\begin{split}
&a_{\eps}= a\ast \varphi_{\omega(\eps)},\,
b_{1,\eps}= b_1\ast \varphi_{\eps},\,
b_{2,\eps}= b_2\ast \varphi_{\eps},\,
b_{3,\eps}= b_3\ast \varphi_{\eps},\,\\
&f_\eps = f(t,\cdot)\ast\varphi_\eps,\,
g_{0,\eps}=g_0\ast\varphi_\eps,\,
g_{1,\eps}=g_1\ast\varphi_\eps,
\end{split}
\] 
where $\omega(\eps)^{-1}$ is a scale of logarithmic type and the following Levi condition
\begin{align*}  
b_{1,\eps}(x)^2 \le \ln(\eps^{-1}) a_{\eps}(x),
\end{align*}
is fulfilled for $\eps\in(0,1]$, then the Cauchy problem has a very weak solution of Sobolev order $0$.
\end{theorem}

\subsection{Case 2}
We assume that our coefficients $a, b_1, b_2, b_3\in \E'(\R)$ and $a\ge 0$. From Proposition \ref{prop_Glaeser_mollified} (ii), $M_{1,\eps}=\omega(\eps)^{-2-m_{a}} \sum_{\beta\le m_{a} } \Vert  g_{a,\beta} \Vert_\infty  \Vert \partial^{\beta+2}\varphi \Vert_{L^1}$, where $m_{a} \in \N$ and $g_{a,\beta} \in C_c(\R)$ come from the structure of $a$. Analogously, from Proposition \ref{lemma_mollification} (ii),  $\Vert a_{\eps}\Vert_\infty\le  \omega(\eps)^{-m_{a}} \sum_{\beta\le m_{a} } \Vert  g_{a,\beta} \Vert_\infty \Vert \partial^{ \beta}\varphi \Vert_{L^1}$, $\Vert b_{2,\eps}\Vert_\infty\le  \omega(\eps)^{-m_{b_2}} \sum_{\beta\le m_{b_2} } \Vert  g_{b_2,\beta} \Vert_\infty \Vert \partial^{ \beta}\varphi \Vert_{L^1}$, $\Vert b_{3,\eps}\Vert_\infty\le  \omega(\eps)^{-m_{b_3}} \sum_{\beta\le m_{b_3} } \Vert  g_{b_3,\beta} \Vert_\infty \Vert \partial^{ \beta}\varphi \Vert_{L^1}$. The above calculations in combination with \eqref{C_2_3x3_Case 1}, \eqref{est_U_2_3x3_Case1} and Proposition \ref{prop_rhs_mod} show that we have proven the following theorem.

%
%

%

\begin{theorem}
\label{theo_main_case_2}
Let us consider the Cauchy problem
\beq
\label{CP_1_ex_case2}
\begin{split}
\partial_t^2u-a(x)\partial^2_{x} u+ b_1(x)\partial_x u+ b_2(x)\partial_t u + b_3(x) u&=f(t,x),\\
u(0,x)&=g_0,\\
\partial_tu(0,x)&=g_1,
\end{split}
\eeq
where $t\in[0,T]$, $x\in\R$.
Assume the following set of hypotheses (Case 2):
\begin{itemize}
\item[(i)] the coefficients $a, b_1, b_2, b_3\in \E'(\R)$ and $a\ge 0$,
\item[(ii)] $f\in C([0,T],\mathcal{E}'(\R))$,
\item[(iii)] $g_0,g_1\in  \mathcal{E}'(\R)$.
\end{itemize}
If we regularise 
\[
\begin{split}
&a_{\eps}= a\ast \varphi_{\omega(\eps)},\,
b_{1,\eps}= b_1\ast \varphi_{\omega(\eps)},\,
b_{2,\eps}= b_2\ast \varphi_{\omega(\eps)},\,
b_{3,\eps}= b_3\ast \varphi_{\omega(\eps)},\,\\
&f_\eps = f(t,\cdot)\ast\varphi_\eps,\,
g_{0,\eps}=g_0\ast\varphi_\eps,\,
g_{1,\eps}=g_1\ast\varphi_\eps,
\end{split}
\] 
where $\omega(\eps)^{-1}$ is a scale of logarithmic type and the following Levi condition
\begin{align*}  
b_{1,\eps}(x)^2 \le \ln(\eps^{-1}) a_{\eps}(x),
\end{align*}
is fulfilled for $\eps\in(0,1]$, then the Cauchy problem has a very weak solution of Sobolev order $0$.
\end{theorem}

\subsection{Case 3}
We assume that our coefficients $a, b_1, b_2, b_3\in B^{\infty}(\R)$ and $a(x)\ge 0$ for all $x\in\R$. This corresponds to the classical case in Section \ref{section_classical}. In the next proposition we prove that the classical Levi condition on $b_1$ can be transferred on the regularised coefficient $b_{1,\eps}$.  
\begin{proposition}\label{prop_consistency}
Let\[
\begin{split}
\partial_t^2u-a(x)\partial^2_{x} u+ b_1(x)\partial_x u+ b_2(x)\partial_t u + b_3(x) u&=f(t,x),\\
u(0,x)&=g_0,\\
\partial_tu(0,x)&=g_1,
\end{split}
\]
where $a, b_1, b_2, b_3\in B^{\infty}(\R)$ and $a \geq 0$. Suppose the Levi condition
\[	 
b_1(x)^2 \le M_2 a(x), 
\]
holds for some $M_2>0$. Then the Levi condition also holds for $b_{1,\eps}$, i.e.,  
\[	 
b_{1,\eps}(x)^2 \le \tilde{M_2} a_{\eps}(x),
\]
for some constant $\tilde{M_2}$, independent of $\eps$.
\end{proposition}

\begin{proof}
We begin by writing 
\beq
 \label{estimate1}
 \begin{split}
&b_{1,\eps}(x)^2 = \left|\, \int\displaylimits_{\supp(\varphi_{\omega(\eps)})} b_{1}(x-\xi)\varphi_{\omega(\eps)} (\xi) d\xi \right|^2 \le \left(\,\int\displaylimits_{\supp(\varphi_{\omega(\eps)})} \left| b_{1}(x-\xi) \right|\varphi_{\omega(\eps)} (\xi)  d\xi \right)^2 \\
&\le \left(\,\int\displaylimits_{\supp(\varphi_{\omega(\eps)})} M^{\frac{1}{2}}_2 a^{\frac{1}{2}}(x-\xi) \varphi_{\omega(\eps)} (\xi)  d\xi \right)^2  \le \mu (\supp(\varphi_{\omega(\eps)}))\int\displaylimits_{\supp(\varphi_{\omega(\eps)})} M_2 a(x-\xi) \varphi_{\omega(\eps)} (\xi)^2  d\xi  
\end{split}
\eeq
where in the last step we used Hölder's inequality.  The right-hand side of \eqref{estimate1} can be estimated by 
\begin{align*}
&= \omega(\eps)C M_2 \int\displaylimits_{\supp(\varphi_{\omega(\eps)})}  a(x-\xi) \varphi_{\omega(\eps)} (\xi)^2  d\xi\\
&= \omega(\eps)C M_2 \int\displaylimits_{\supp(\varphi_{\omega(\eps)})}  a(x-\xi) \varphi_{\omega(\eps)} (\xi)\frac{1}{\omega(\eps)}\varphi \left(\frac{\xi}{\omega(\eps)}\right)  d\xi\\
&\le C M_2 \int\displaylimits_{\supp(\varphi_{\omega(\eps)})}  a(x-\xi) \varphi_{\omega(\eps)} (\xi)\tilde{C} d\xi = \tilde{M_2} a\ast \varphi_{\omega(\eps)}(x) = \tilde{M_2} a_{\eps}(x),
\end{align*}
for some constant $C>0$ independent of $\eps$.
\end{proof}

From Proposition \ref{prop_Glaeser_mollified} (iii), $M_{1}= \Vert\partial^2 a\Vert_\infty$ and $\Vert a_{\eps}\Vert_\infty\le  \Vert a\Vert_\infty$, $\Vert b_{2,\eps}\Vert_\infty\le  \Vert b_{2}\Vert_\infty$, $\Vert b_{3,\eps}\Vert_\infty\le  \Vert b_{3}\Vert_\infty$ from Proposition \ref{lemma_mollification} (iii). The above calculations in combination with \eqref{C_2_3x3_Case 1}, \eqref{est_U_2_3x3_Case1}, Proposition \ref{prop_rhs_mod} and Proposition \ref{prop_consistency} show that we have proven the following theorem.
\begin{theorem}
\label{theo_main_case_3}
Let us consider the Cauchy problem
\beq
\label{CP_1_ex_case3}
\begin{split}
\partial_t^2u-a(x)\partial^2_{x} u+ b_1(x)\partial_x u+ b_2(x)\partial_t u + b_3(x) u&=f(t,x),\\
u(0,x)&=g_0,\\
\partial_tu(0,x)&=g_1,
\end{split}
\eeq
where $t\in[0,T]$, $x\in\R$.
Assume the following set of hypotheses (Case 3):
\begin{itemize}
\item[(i)] the coefficients $a, b_1, b_2, b_3\in  B^{\infty}(\R)$ and $a(x)\ge 0$ for all $x\in\R$,
\item[(ii)] $f\in C([0,T],C_c^{\infty}(\R))$,
\item[(iii)] $g_0,g_1\in  C_c^{\infty}(\R)$,
\end{itemize}
and that the Levi condition $b_1(x)^2\le M_2a(x)$ holds uniformly in $x$. If we regularise 
\[
\begin{split}
&a_{\eps}= a\ast \varphi_{\eps},\,
b_{1,\eps}= b_1\ast \varphi_{\eps},\,
b_{2,\eps}= b_2\ast \varphi_{\eps},\,
b_{3,\eps}= b_3\ast \varphi_{\eps},\,\\
&f_\eps = f(t,\cdot)\ast\varphi_\eps,\,
g_{0,\eps}=g_0\ast\varphi_\eps,\,
g_{1,\eps}=g_1\ast\varphi_\eps,
\end{split}
\] 
then the Cauchy problem has a very weak solution of Sobolev order $0$.
\end{theorem}
Theorem \ref{theo_main_case_3} proves the existence of a very weak solution when the equation coefficients are regular. Now, we want to prove that in this case every very weak solution converges to the classical solution as $\eps\to 0$. This result  holds independently of the choice of regularisation, i.e., mollifier and scale.

\begin{theorem}
\label{theo_consistency}
Let 
\beq
\label{CP_cons}
\begin{split}
\partial_t^2u-a(x)\partial^2_{x} u+ b_1(x)\partial_x u+ b_2(x)\partial_t u + b_3(x) u&=f(t,x),\quad t\in[0,T], x\in\R,\\
u(0,x)&=g_0,\\
\partial_tu(0,x)&=g_1,
\end{split}
\eeq
where $a, b_1, b_2, b_3\in B^{\infty}(\R)$ and $a \geq 0$. We also assume that all the functions involved in the system are real-valued. Let $g_0$ and $g_1$ belong to $C_c^{\infty}(\R)$ and $f \in C([0,T];C_c^{\infty}(\R))$. Suppose the Levi condition 
\begin{align*}
\label{THM:lot2_general}
b_{1}(x)^2 \le M_2 a(x)
\end{align*}
holds. Then, every very weak solution $(u_\eps)_\eps$ of $u$ converges in  $C([0,T], L^2(\R))$ as $\eps\to0$ to the unique classical solution of the Cauchy problem \eqref{CP_cons};
\end{theorem}

\begin{proof}
 Let $\wt{u}$ be the classical solution. By definition we have
\beq
\label{CP_class}
\begin{split}
\partial_t^2 \wt{u}-a(x)\partial^2_{x} \wt{u}+ b_{1}(x)\partial_x \wt{u}+ b_{2}(x)\partial_t \wt{u} + b_{3}(x) \wt{u}&=f(t,x),\quad t\in[0,T], x\in\R,\\
\wt{u}(0,x)&=g_{0},\\
\partial_t \wt{u}(0,x)&=g_{1}.
\end{split}
\eeq 
By the finite speed of propagation of hyperbolic equations, $\wt{u}$ is compactly supported with respect to $x$. As before, we regularise \eqref{CP_cons} and by Proposition \ref{prop_consistency}, the Levi condition also holds for the regularised coefficients with a constant independent of $\eps$. Note that we do not need to regularise the initial data as they are already smooth. Hence, there exists a very weak solution $(u_\eps)_\eps$ of $u$ such that
\beq
\label{CP_class_2}
\begin{split}
\partial_t^2 u_{\eps}-a_{\eps}(x)\partial^2_{x} u_{\eps}+ b_{1,\eps}(x)\partial_x u_{\eps}+ b_{2,\eps}(x)\partial_t u_{\eps} + b_{3,\eps}(x) u_{\eps}&=f_{\eps}(t,x),\quad t\in[0,T], x\in\R,\\
u_{\eps}(0,x)&=g_{0},\\
\partial_t u_{\eps}(0,x)&=g_{1},
\end{split}
\eeq
for suitable embeddings of the coefficients.  We can therefore rewrite \eqref{CP_class} as
\beq
\label{CP_class_3}
\begin{split}
\partial_t^2 \wt{u}-a_{\eps}(x)\partial^2_{x} \wt{u}+ b_{1,\eps}(x)\partial_x \wt{u}+ b_{2,\eps}(x)\partial_t \wt{u} + b_{3,\eps}(x) \wt{u}&=f_{\eps}(t,x)+n_\eps(t,x),\\
\wt{u}(0,x)&=g_{0},\\
\partial_t \wt{u}(0,x)&=g_{1},
\end{split}
\eeq 
where
\[
n_\eps(x)=(f-f_{\eps})+(a-a_{\eps})\partial^2_x \wt{u}-(b_1-b_{1,\eps})\partial_x \wt{u}-(b_2-b_{2,\eps})\partial_t\wt{u}-(b_3-b_{3,\eps})\wt{u},
\]
and   $n_\eps \in C([0,T], L^2(\R))$ and converges to $0$ in this space. Note that here we have used the fact that all the terms defining $n_\eps$ have $L^2$-norm which can be estimated by $\eps$ uniformly with respect to $t\in[0,T]$. Hence, $\Vert n_\eps\Vert_{L^2}=O(\eps)$ and $n_\eps$ converge to $0$ in $C([0,T], L^2(\R))$ as $\eps\to 0$.  From \eqref{CP_class_3} and \eqref{CP_class_2} we obtain that $\wt{u}-u_\eps$ solves the equation
\[
\begin{split}
\partial_t^2 (\wt{u}-u_\eps)&-a_{\eps}(x)\partial^2_{x} (\wt{u}-u_\eps)\\
&+ b_{1,\eps}(x)\partial_x (\wt{u}-u_\eps)+ b_{2,\eps}(x)\partial_t (\wt{u}-u_\eps) + b_{3,\eps}(x) (\wt{u}-u_\eps)=n_\eps(t,x),
\end{split}
\]
which fulfils the initial conditions
\[
\begin{split}
(\wt{u}-u_\eps)(0,x)&=0,\\
(\partial_t \wt{u}- \partial_t u_\eps)(0,x)&=0.
\end{split}
\]
Following the energy estimates Section \ref{section_classical} and as in the proof of Theorem \ref{theo_main_case_3}, after reduction to a system,  we have an estimate of  $\Vert\wt{u}- u_{\eps}\Vert_{L^2}^2$ as in \eqref{est_U_2_3x3_Case1}, in terms of  $(\wt{u}-u_\eps)(0,x)$, $(\partial_t \wt{u}- \partial_t u_\eps)(0,x)$ and the right-hand side $n_\eps(t,x)$. In particular, since $(\wt{u}-u_\eps)(0,x)=(\partial_t \wt{u}- \partial_t u_\eps)(0,x)=0$ and noting that the constant $C_\eps$ in \eqref{est_U_2_3x3_Case1} is independent of $\eps$ in this case,  we simply get
\begin{align*} \label{impor_est_const}
\Vert\wt{u}- u_{\eps}\Vert_{L^2}^2\le C \int_{0}^t\Vert n_{\eps}(s)\Vert_{L^2}^2\,ds.
\end{align*} 
Since   $n_\eps\to 0$ in  $ C([0,T], L^2(\R))$, we conclude that $u_\eps\to \wt{u}$ in  $ C([0,T], L^2(\R))$. Note that our argument is independent of the choice of the regularisation of the coefficients and the right-hand side. 

\end{proof}
Concluding, we have proven the existence of very weak solutions of Sobolev order $0$ since in the next sections we will mainly work on $L^2$-norms. However, very weak solutions of higher Sobolev order can also be obtained with a suitable choice of nets involved in the regularisation process by following the techniques developed in Section \ref{sec_Oleinik}. In analogy to Section 7 in \cite{DGL:22} it is immediate to see from the estimate \eqref{est_U_2_3x3_Case1} that perturbing the equation's coefficients, right-hand side and initial data with negligible needs leads to a negligible change in the solution, where with negligible we intend negligible net of Sobolev order $0$. So, we can conclude that our Cauchy problem is very weakly well-posed.

\section{Toy models and numerical experiments} \label{Section_piecewise_sol}
The final section of this paper is devoted to a wave equation toy model with discontinuous space-dependent coefficient. This complements the study initiated in \cite{DGL:22} where the equation's coefficient was depending on time only. In detail, we consider the Cauchy problem 
\beq
\label{CP_heaviside_Deguchi}
\begin{split}
\partial_t^2u(t,x)-\partial_x (H(x)  \partial_{x} u(t,x))&=0,\qquad t\in[0,T],\, x\in\R,\\
u(0,x)&=g_0(x),\qquad x\in\R,\\
\partial_t u(0,x)&=g_1(x), \qquad x\in\R,
\end{split}
\eeq
where $g_0,\, g_1 \in C_c^{\infty}(\R)$ and $H$ is the Heaviside function with jump at $x=0$. Classically, we can solve this Cauchy problem piecewise to obtain the piecewise distributional solution $\bar{u}$. First for $x<0$, then for $x>0$ and then combining these two solutions together under the transmission conditions that $\bar{u}$, $\partial_t \bar{u}$, $H\partial_x \bar{u}$ are continuous across $x=0$, as in \cite{GO:14}. Problems of this type have been investigated in \cite{DO:16} for a jump function between two positive constants. Here, we work with the Heaviside function which is equal to zero before the jump and then identically $1$.  

In the next subsection we describe how the classical piecewise distributional solution $\bar{u}$ of the Cauchy problem \eqref{CP_heaviside_Deguchi} is defined.

\subsection{Classical piecewise distributional solution}
 
In the region defined by $x<0$, our equation becomes 
\[
\partial^2_tu=0.
\]
Integrating once we get $\partial_tu=g_1(x)$. Integrating once more, we get $u(t,x)=tg_1(x)+g_0(x)$. From this we obtain that $\partial_tu(t,x)=g_1(x)$ and $\partial_xu(t,x)=tg_1'(x)+g_0'(x)$.
We first do a transformation into a system
\[v=w=\partial_t u, \quad v_0(x)=w_0(x)=g_1(x). \]

Note that from the equation we get that $\partial_t v =\partial_t w=0 $ and hence we have that $v(t,x)=w(t,x)=g_1(x)$ for $x<0$.

In the region $x>0$, the equation becomes 
\[
\partial^2_tu-\partial_x^2u=0.
\]
We first do a transformation into a system
\begin{align*}
v=(\partial_t -\partial_x)u, \quad &w=(\partial_t +\partial_x)u ,\\
v_0(x)=g_1(x)-g_0'(x),\quad &w_0(x)=g_1(x)+g_0'(x).
\end{align*}
From this we can recover that
\[
\partial_t u=\frac{1}{2}(v+w), \quad \partial_x u=\frac{1}{2}(w-v).
\]

Note that the system satisfies
\begin{align*}
&(\partial_t+\partial_x)v=0 \\
&(\partial_t-\partial_x)w=0 \\
&v(0,x)=v_0(x), w(0,x)=w_0(x).
\end{align*}

We therefore have that $v(t,x)=v_0(x-t)$ for $x \ge t $ and $w(t,x)=w_0(x+t)$ when $0 \le x < t $   or $x \ge t $ (i.e. $x\ge 0$).

We now compare the values of $v$ and $w$ on both sides of $x=0$. We denote by $v_-,w_-$ and $v_+,w_+$ the values of $v,w$ for $x<0$ and $x>0$, respectively.
Using the conditions that the values of $u_t$ and $u_x$ should not jump across $x=0$, we get that
\begin{align*}
v_+ + w_+ &= v_- + w_- (=2v_-=2w_-) \\
w_+ -v_+ &= 0
\end{align*}
along $x=0$.
Substituting the value of $w_+$ and $v_-$ we get

\begin{align*}
v_+ + g_1(t) +g_0'(t) &= 2g_1(0) \\
g_1(t) +g_0'(t) -v_+ &= 0.
\end{align*}

Rearranging with respect to $v_+$, we have

\begin{align*}
v_+  &= 2g_1(0) - g_1(t) - g_0'(t) \\
v_+ &=  g_1(t) +g_0'(t) .
\end{align*}

This gives us the condition 
\beq \label{condition2}
 g_1(t) +g_0'(t) -g_1(0) =0.
\eeq
This condition on the initial data can also be found in \cite[Section 2]{B:84}. Concluding we get that,

\begin{equation*} 
v(t,x)=
\begin{cases}
g_1(x), \quad  &\text{for } x < 0,\\
2g_1(0)-g_1(t-x) -g_0'(t-x) , \quad &\text{for } 0\le x <t  ,\\
v_0(x-t) = g_1 (x-t) -g_0'(x-t) , \quad &\text{for } x \ge t,
\end{cases}
\end{equation*}

and

\begin{equation*} 
w(t,x)=
\begin{cases}
g_1(x), \quad  &\text{for } x < 0,\\
g_1(x+t)+g_0'(x+t), \quad &\text{for } x \ge 0 .
\end{cases}
\end{equation*}

Using that $\partial_t u =\frac{v+w}{2}$,  we get that 
\beq \label{solution_t_heaviside}
u_t(t,x)=
\begin{cases}
g_1(x), \quad  &\text{for } x < 0,\\
\frac{1}{2} (g_0'(x+t)-g_0'(t-x)) +\frac{1}{2} (g_1(x+t)-g_1(t-x)) +g_1(0) , \quad &\text{for } 0\le x < t ,\\
\frac{1}{2} (g_0'(x+t)-g_0'(x-t))+\frac{1}{2} (g_1(x+t)+g_1(x-t))  , \quad &\text{for } x \ge t.
\end{cases}
\eeq

Note that \eqref{solution_t_heaviside} is continuous on the lines $x=0$ and $x=t$.  

Integrating and using the initial conditions, we conclude that the solution of \eqref{CP_heaviside_Deguchi} for $t\in[0,T] $ is

\beq \label{solution_heaviside}
\bar{u}(t,x)=
\begin{cases}
g_0(x)+tg_1(x), \quad &\text{for } x < 0,\\
\frac{1}{2}(g_0(x+t)-g_0(t-x))+\frac{1}{2}\int_{t-x}^{x+t}g_1(s)ds +g_0(0)+(t-x)g_1(0), \quad &\text{for } 0 \le x < t,\\
\frac{1}{2}(g_0(x+t)+g_0(x-t))+\frac{1}{2}\int_{x-t}^{x+t}g_1(s)ds, \quad &\text{for } x \ge t. 
\end{cases}
\eeq
Note that \eqref{solution_heaviside} is continuous on the lines $x=0$ and $x=t$ and hence $u \in C^1([0,T]\times \R )$.
Furthermore we have that $\bar{u}(0,x)=g_0(x)$ and $\partial_t\bar{u}(0,x)=g_1(x)$.

The Cauchy problem \eqref{CP_heaviside_Deguchi} can also be studied via the very weak solutions approach. In the sequel, we prove that every weak solution obtained after regularisation will recover the distributional solution $\bar{u}$ in the limit as $\eps\to 0$.
\subsection{Consistency of the very weak solution approach with the classical one}
Let
\beq
\label{CP_gen_reg}
\begin{split}
\partial_t^2u_{\eps}(t,x)-\partial_x (H_{\eps}(x) \partial_x u_{\eps}(t,x))&=0,\qquad t\in[0,T], x\in\R,\\
u_{\eps}(0,x)&=g_0(x),\qquad x\in\R,\\
\partial_t u_{\eps}(0,x)&=g_1(x), \qquad x\in\R,
\end{split}
\eeq
where $H_{\eps}=H\ast \varphi_{\omega(\eps)}$ with $\varphi_{\omega(\eps)}(x)=\eps^{-1}\varphi(x/\eps)$ and $\varphi$ is a mollifier ($\varphi\in C^\infty_c(\R^n)$, $\varphi\ge 0$ with $\int\varphi=1$). Assume that $\omega(\eps)$ is a positive scale with  $\omega(\eps)\to 0$ as $\eps\to 0$ and that the initial data belong to $C^\infty_c(\R)$.

\begin{theorem} \label{thm_convergence}
Every very weak solution of Cauchy problem \eqref{CP_heaviside_Deguchi} converges to the piecewise distributional solution $\bar{u}$, as $\eps \to 0$.  
\end{theorem}

\begin{proof}
Let $(u_\eps)_\eps$ be a very weak solution of the Cauchy problem \eqref{CP_heaviside_Deguchi}, i.e. a net of solutions of the Cauchy problem \eqref{CP_gen_reg} which fulfills moderate $L^2$-estimates.  Since the initial data are compactly supported, from the finite speed of propagation,  it is not restrictive to assume that $(u_\eps)$ vanishes for $x$ outside some compact set K, independently of $t \in [0,T]$ and $\eps\in (0,1]$.  By construction we have that $u_\eps\in C^{\infty}([0,T], C_c^{\infty}(\R))$. We now argue as in the proof of Proposition 6.8 in \cite{GO:14}. We first  multiply the equation $u_{\eps,tt} - (H_\eps u_{\eps,x})_x=0 $ by $u_{\eps,t}$ and integrate to get the conservation law
\beq \label{conservation1}
\frac{1}{2}\frac{d}{dt} \int_\R \left( |u_{\eps,t}|^2 +H_{\eps} |u_{\eps,x}|^2\right) dx=0 \quad \text{ for } t \in [0,T].
\eeq 
When we differentiate the wave equation once with respect to time, multiply by $u_{\eps,tt}$ and integrate, we get
\beq \label{conservation2}
\frac{1}{2}\frac{d}{dt} \int_\R \left( |u_{\eps,tt}|^2 +H_{\eps} |u_{\eps,xx}|^2\right) dx=0 \quad \text{ for } t \in [0,T].
\eeq 
The regularity of the initial data and \eqref{conservation1}, imply that $\int_\R|u_{\eps,t}(t,x)|^2dx$ is bounded on $[0,T]$. We therefore have that $(u_\eps)_\eps$ is bounded in $C^1([0,T], L^2(\R))$. By the Mean Value Theorem, $(u_\eps)_\eps$ is equicontinuous in  $C([0,T], L^2(\R))$. Since $(u_\eps)_\eps$ is also bounded in $C([0,T], L^2(\R))$, by Arzela-Ascoli, $(u_\eps)_\eps$ is relatively compact in $C([0,T], L^2(\R))$.  From \eqref{conservation2}, we get that $\int_\R|u_{\eps,tt}(t,x)|^2dx$ is also bounded on $[0,T]$. By the same argument as above, we obtain that $(u_\eps)_\eps$ is relatively compact in $C^1([0,T], L^2(\R))$.
We have that $(u_\eps)_\eps$ and $(u_{\eps,t})_\eps$ are relatively compact in $C([0,T], L^2(\R))$. Furthermore, we have that $(u_{\eps,tt})_\eps$ is bounded in $ L^2([0,T] \times \R)$. Therefore, for a suitable subsequence, we have that there exist $U, U_1 \in  C([0,T], L^2(\R))$ and $U_2 \in L^2([0,T] \times \R)$ such that 
\begin{align*}
& u_\eps \rightarrow U, \, u_{\eps,t} \rightarrow U_1 \text{ strongly in } C([0,T], L^2(\R)) \subset C([0,T], \mathcal{D}'(\R)) \text{ and} \\
& u_{\eps,tt} \rightarrow U_2 \text{ weakly in } L^2([0,T] \times \R).
\end{align*}
From a routine argument we therefore have that $U_1=U_t$ and $U_2=U_{tt}$. The equation gives us that $(H_\eps u_{\eps,x})_x$ is also bounded in $ L^2([0,T] \times \R)$ and also converges weakly to  $U_2  \in L^2([0,T] \times \R)$. Due to the support properties, the convolution $V(t,\cdot)= U_2 \ast H $ exists. Therefore  $(H_\eps u_{\eps,x})_x\ast H \rightarrow U_2 \ast H =V(t,\cdot)$ weakly in  $ L^2([0,T] \times \R)$. Hence  $(H_\eps u_{\eps,x})_x\ast H \rightarrow V(t,\cdot)$ in  $ C([0,T], \mathcal{D}'(\R))$. Putting the derivative on the Heaviside, we get that $H_\eps u_{\eps,x}\ast \delta= H_\eps u_{\eps,x}\rightarrow  V$  in  $ C([0,T], \mathcal{D}'(\R))$.

It follows that 
\beq \label{almost_wave}
U_{tt}-V_x=0 \quad \quad \text{in }  C([0,T], \mathcal{D}'(\R)).
\eeq
Note that for every given $x>0$, we can find $\eps>0$ small enough such that  $H_\eps (x) = 1$. Similarly, for every given $x<0$,  $H_\eps (x) = 0$  for $\eps>0$ sufficiently small. Hence, taking $\varphi \in C^\infty_c ((-\infty,0))$, $\varphi=0$ in an open neighbourhood of $\supp(H_\eps u_{\eps,x})$ for $\eps$ small enough and therefore $\langle H_\eps u_{\eps,x},\varphi\rangle=0$. Analogously, taking $\varphi \in C^\infty_c ((0,\infty))$ we have that $\langle H_\eps u_{\eps,x},\varphi\rangle=\langle  u_{\eps,x},\varphi\rangle$, for $\eps$ small enough. Thus, in the sense of distributions, we have
\beq \label{equationV}
V(t,\cdot)=
\begin{cases}
0, &\text{ on }(-\infty,0), \\
U_x, &\text{ on }(0,\infty),
\end{cases}\quad\text{ and hence } \quad
V_x(t,\cdot)=
\begin{cases}
0, &\text{ on }(-\infty,0), \\
U_{xx}, &\text{ on }(0,\infty),
\end{cases}
\eeq
for almost all $t$. It follows, from \eqref{almost_wave} and \eqref{equationV}, that $U$ is a distributional solution to the Cauchy Problem \eqref{CP_heaviside_Deguchi} on both sides of $x=0$. Since $U \in C([0,T],\mathcal{D}'(\R))$, so is $(H U_{x})_x$ and hence from the equation, so is also $U_{tt}$. That is, $U$ is the unique piecewise distributional solution to the Cauchy problem \eqref{CP_heaviside_Deguchi}. Therefore $U=\bar{u}$. Since $(u_\eps)_\eps$ is equicontinuous in  $C^1([0,T], L^2(\R))$, we have that the whole net $(u_\eps)_\eps$ converges to $U=\bar{u}$.

\end{proof}

\subsection{Numerical model}
We conclude the paper by investigating the previous toy-model numerically. We provide numerical investigations of more singular toy-models by replacing the Heaviside coefficient with a delta of Dirac or an homogeneous distribution. Our analysis is aimed to have a better understanding of the corresponding very weak solutions as $\eps\to 0$.
\subsubsection{Heaviside}
We solve the Cauchy problem \eqref{CP_heaviside_Deguchi} numerically using the Lax–Friedrichs method \cite{LV92} after transforming it to an equivalent first-order system. We consider $t \in [0,2]$, $x \in [-4,4]$, and compactly supported initial conditions 
\[
g_0(x) = 
\begin{cases*}
-x^4(x-1)^4(x+1)^4, & if $\lvert x \rvert < 1$,  \\
                  0, & if $\lvert x \rvert \ge 1,$
\end{cases*}
\]
\[
g_1(x) = 
\begin{cases*}
4x^3(x-1)^4(x+1)^4+4x^4(x-1)^3(x+1)^4+4x^4(x-1)^4(x+1)^3, & if $\lvert x \rvert < 1$,  \\
                  0, & if $\lvert x \rvert \ge 1,$
\end{cases*}
\]
satisfying condition \eqref{condition2}. For the space and time discretisation, we fix the discretisation steps $\Delta x=\Delta t= 0.0005$. For different values of $\eps$, we compute the numerical solutions $u_\eps$ and we compare them with the piecewise distributional solution $\bar{u}$ obtained  in Section \ref{Section_piecewise_sol}. In particular, we compute the norm $\|u_\eps - \bar{u} \|_{L^2([-4,4])}$ at $t=2$ and show that it tends to $0$ for $\eps \to 0$ as in Theorem \ref{thm_convergence}.

We consider the mollifier $\varphi_{\eps}(x)=\frac{1}{\eps}\,\varphi(\frac{x}{\eps})$ with
\[
\varphi(x) = 
\begin{cases*}
\frac{1}{0.443994} e^{\frac{1}{(x^2 -1)}}, & if $\lvert x \rvert < 1$,  \\
                                        0, & if $\lvert x \rvert \ge 1.$
\end{cases*}
\]
Note that we do not expect that a change in the mollifier will impact our results as we have shown similarly in \cite{DGL:22}.  

Figure \ref{fig:numerics1} shows the exact solution $\bar{u}$ and the solutions $u_\eps$ at $t=2$ for $\eps=0.1$, 0.05, 0.01, 0.005, 0.001, 0.0005 (left), a close-up around $x=0$ (right) and the computed $L^2$ error norm for the various values of $\eps$ (bottom). We can see that the solutions $u_\eps$ better approach the exact solution $\bar{u}$ as $\eps$ is reduced and that the error norm decreases for $\eps \to 0$ as expected.

\begin{figure}[H]
\centerline{%
\includegraphics[width=0.49\textwidth]{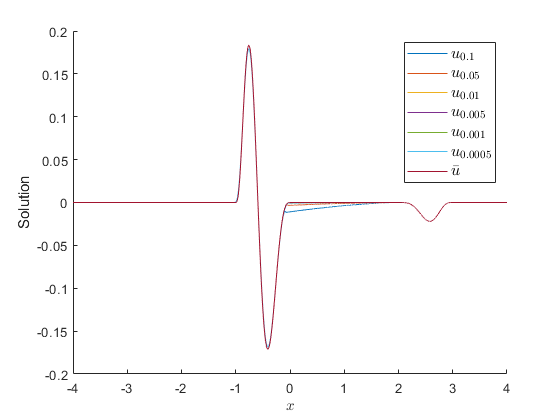}
\hfill
\includegraphics[width=0.49\textwidth]{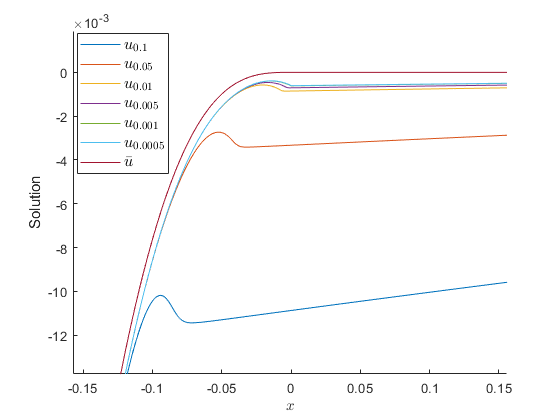}
}
\includegraphics[width=0.49\textwidth]{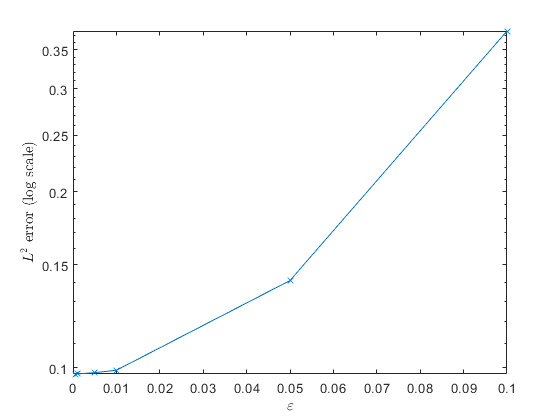}
\caption{Exact solution $\bar{u}$ and numerical approximations $u_\eps$ for $\eps=0.1,0.05,0.01,0.005,0.001,0.0005$ (left); a close-up around $x=0$ (right); norm $\|u_\eps - \bar{u} \|_{L^2([-4,4])}$ at $t=2$ versus $\eps$ (bottom).}
\label{fig:numerics1}
\end{figure}

%
%

\subsubsection{Delta} 
We now focus on the Cauchy problem \eqref{CP_heaviside_Deguchi} where we replace the Heaviside function $H_\eps$ by a regularised delta distribution $\delta_\eps$. We consider the same setting and discretisation as in the previous test, we compute the numerical solutions $u_\eps$ and we calculate the norm $\|u_\eps  \|_{L^2([-4,4])}$ at $t=0.05$ and show that it grows to infinity as $\eps \to 0$. As a result, $u_\eps$ cannot have a limit in $L^2([-4,4])$ since if it did, $\|u_\eps\|_{L^2([-4,4])}$ would have have to converge. 

In Figure \ref{fig:numerics2} we depict the computed $L^2$ norm at $t=0.05$ for $\eps=0.1$, 0.05, 0.01, 0.005, 0.001, 0.0005. As expected, this seems to confirm the blow-up of the $L^2$ norm in the case of coefficients as singular as delta.

\begin{figure}[H]
\centerline{%
\includegraphics[width=0.49\textwidth]{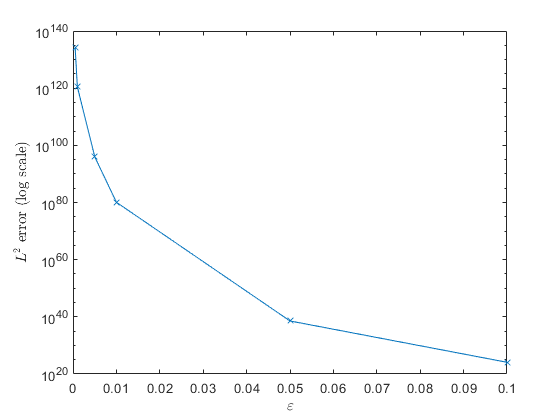}
}
\caption{Norm $\|u_\eps \|_{L^2([-4,4])}$ at $t=0.05$ versus $\eps$.}
\label{fig:numerics2}
\end{figure}

 \subsubsection{Homogeneous distributions} 
We finally study numerically the Cauchy problem \eqref{CP_heaviside_Deguchi} where the coefficient is replaced by  the  regularisation of the homogeneous distribution $\chi^{\alpha}_{+}$, as defined in \cite{HO:15}
\[
\chi^{\alpha}_{+}= \frac{x^{\alpha}_{+}}{\Gamma(\alpha +1)}.
\]
Following \cite{HO:15}, we recall that when $\alpha= 0 $, $\chi^{\alpha}_{+}$ corresponds to the Heaviside function and when  $\alpha= -1 $, $\chi^{\alpha}_{+}$ corresponds to the delta distribution. As before, we consider the same setting and discretisation. For different values of $\alpha$ between 0 and $-1$, we compute the numerical solutions $u_\eps$ and we calculate the norm $\|u_\eps  \|_{L^2([-4,4])}$ at $t=0.05$ and compare how it grows compared to the previous cases. 
In Figure \ref{fig:numerics3} we depict the computed $L^2$ norm at $t=0.05$ for $\eps=0.1$, 0.05, 0.01, 0.005, 0.001, 0.0005 and for $\alpha=0$, $-0.1$, $-0.25$, $-0.5$, $-0.75$, $-0.9$, $-1$. As expected, this seems to confirm a faster blow-up of the $L^2$ norm as $\alpha \to -1$  ($\alpha =-1$ corresponds to the case of delta). Note that for the case $\alpha =-1$, we have used the information from the prior analysis in the previous section. 

\begin{figure}[H]
\centerline{%
\includegraphics[width=0.49\textwidth]{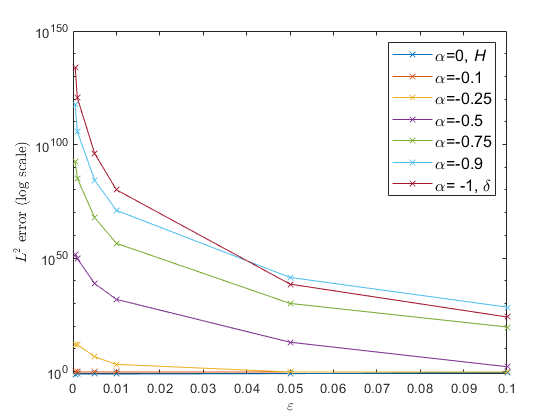}
}
\caption{Norm $\|u_\eps \|_{L^2([-4,4])}$ at $t=0.05$ versus $\eps$, for different values of $\alpha$ between 0 and $-1$.}
\label{fig:numerics3}
\end{figure}

\end{document}